\documentclass[11pt]{article}
\pdfoutput=1
\usepackage[UKenglish]{isodate}
\usepackage[T1]{fontenc}
\usepackage[noBBpl]{mathpazo}
\usepackage{microtype}
\usepackage{multicol}
\usepackage[hidelinks]{hyperref}
\cleanlookdateon
\usepackage[a4paper,margin=2.8cm]{geometry}
\usepackage{stmaryrd}
\SetSymbolFont{stmry}{bold}{U}{stmry}{m}{n}

\usepackage{amsmath, amssymb, amsthm, amsfonts, dsfont, array, float, url, tikz, afterpage}
\usetikzlibrary{arrows.meta, shapes, calc}

\renewcommand{\:}{\colon}
\newcommand{\<}{\langle}
\renewcommand{\>}{\rangle}
\renewcommand{\leq}{\leqslant}
\renewcommand{\geq}{\geqslant}
\newcommand{\Z}{\mathbb{Z}}
\renewcommand{\S}{\mathcal{S}}
\newcommand{\C}{\mathfrak{C}}
\newcommand{\X}{\mathfrak{X}}
\newcommand{\Sym}{\mathrm{Sym}}
\newcommand{\Alt}{\mathrm{Alt}}
\newcommand{\Aut}{\mathrm{Aut}}
\newcommand{\Homeo}{\mathrm{Homeo}}

\newcommand{\Supp}{\mathrm{Supp}}
\newcommand{\Fix}{\mathrm{Fix}}

\newtheoremstyle{defn}{8pt}{4pt}{}{}{\bfseries\boldmath}{.}{0.3em}{} 
\newtheoremstyle{shplain}{8pt}{4pt}{\itshape}{}{\bfseries\boldmath}{.}{0.3em}{} 
\theoremstyle{shplain}
\newtheorem{theorem}{Theorem}[section]
\newtheorem{thm}{Theorem}
\newtheorem{corollary}[theorem]{Corollary}
\newtheorem{cor}[thm]{Corollary}
\newtheorem{proposition}[theorem]{Proposition}
\newtheorem{lemma}[theorem]{Lemma}
\theoremstyle{defn}

\newtheorem{rem}[thm]{Remark}
\newtheorem{example}[theorem]{Example}
\newtheorem{ex}[thm]{Example}
\newtheorem{question}[theorem]{Question}
\newtheorem{que}[thm]{Question}
\newtheorem{definition}[theorem]{Definition}
\newtheorem*{acknowledgements*}{Acknowledgements}

\usepackage{titlesec}
\titlespacing*{\section}{0pt}{\baselineskip}{0pt}
\titlespacing*{\subsection}{0pt}{0.66\baselineskip}{0pt}
\usepackage[shortlabels]{enumitem}
\setlist{leftmargin=0.8cm,topsep=0pt,itemsep=-2pt}
\setlist[enumerate]{label=\rm{(\roman*)}}
\numberwithin{equation}{section}
\usepackage[tableposition=top, figureposition=bottom, skip=0.5\baselineskip]{caption}

\renewenvironment{thebibliography}[1]
{ \begin{oldthebibliography}{#1}
  \setlength{\parskip}{0pt}
  \setlength{\itemsep}{2pt plus 0.3ex}
  \bgroup\footnotesize }
{ \egroup \end{oldthebibliography} }

\makeatletter
\renewenvironment{proof}[1][\proofname]{\par
  \pushQED{\qed}%
  \normalfont
  \topsep2pt \partopsep1pt 
  \trivlist
  \item[\hskip\labelsep
        \itshape
    #1\@addpunct{.}]\ignorespaces
}{%
  \popQED\endtrivlist\@endpefalse
  \addvspace{6pt plus 6pt}
}
\makeatother

\makeatletter
\g@addto@macro\normalsize{%
  \setlength\abovedisplayskip{0.4\baselineskip plus 0.4\baselineskip}
  \setlength\belowdisplayskip{0.4\baselineskip plus 0.4\baselineskip}
  \setlength\abovedisplayshortskip{-0.3\baselineskip}
  \setlength\belowdisplayshortskip{0.4\baselineskip plus 0.4\baselineskip}
}
\makeatother

\begin{document}

\begin{center} 
{\LARGE \textbf{Generating simple vigorous groups}}              \\[11pt]
{\Large Collin Bleak, Casey Donoven, Scott Harper \& James Hyde} \\[22pt]
\end{center}

\begin{center}
\begin{minipage}{0.85\textwidth}
\small 
The simple vigorous groups form a broad class of groups of homeomorphisms of Cantor space that includes Thompson's group $V$, its various generalisations and many others such as Nekrashevych’s groups of dynamical origin.
Bleak, Elliott and Hyde (2024) proved that every finitely generated simple vigorous group is $2$-generated, and, in this paper, we give several strong generation results for this class of simple groups. 
For example, we prove that if $G$ is a finitely generated simple vigorous group, then $G$ is generated by three involutions, $G$ is generated by an element of order $m$ and an element of order $n$ for any choice of $m \geq 2$ and $n \geq 3$, $G$ has a minimal generating set of size $k$ for all $k \geq 2$, every nontrivial element of $G$ is contained in a generating pair and the direct power $G^n$ is $2$-generated for all $n$.
These results are analogous to well-known results for finite simple groups, but of course the proofs in this context are quite different.
One consequence of our results is that Thompson's group $V$ is $(2,3)$-generated, which answers a question of Sapir (2017).
Another consequence is that every finitely generated group quasi-isometrically embeds in a $(2, 3)$-generated simple group, strengthening theorems of Hall (1974) and Bridson (1998).
All of our proofs are constructive, and we establish several new generation criteria for these groups, which we expect to be of wider interest, even just for Thompson's group $V$.
\par
\end{minipage}
\end{center}


\section{Introduction} 
\label{s:intro}

Bleak, Elliott and Hyde \cite{BleakElliottHyde24} studied the groups of homeomorphisms of Cantor space that satisfy the dynamical condition of being \emph{vigorous} (defined below). 
The class of simple vigorous groups includes Thompson's group $V$ and all known simple groups generalising $V$, such as the Higman--Thompson groups $V_n'$ and the Brin--Thompson groups $nV$, together with other topical families of groups such as Nekrashevych's simple groups of dynamical origin. 
These groups have some remarkable properties, with connections to the word problem  \cite{McKenzieThompson73, Thompson76}, group cohomology \cite{BrownGeoghegan85} and dynamical systems \cite{GhysSergiescu87}. 
Moreover, $V$ and its subgroup $T$ were the first known examples of infinite finitely presented simple groups.
In \cite{BleakElliottHyde24}, Bleak, Elliott and Hyde proved that every finitely generated simple vigorous group is $2$-generated.

In a different direction, in light of the Classification of Finite Simple Groups, every \emph{finite} simple group is known to be $2$-generated \cite{AschbacherGuralnick84}.
Except for some undiscovered sporadic groups, this was proved by Steinberg in 1962 \cite{Steinberg62}. 
Since then an influential line of research has emerged showing that the finite simple groups actually have some striking properties that are much stronger than just admitting a generating pair (see Burness' survey \cite{Burness19}). 
For example, Steinberg asked in his 1962 paper whether every nontrivial element of a finite simple group $G$ is contained in a generating pair, a property known as \emph{$\frac{3}{2}$-generation}.
This was proved by Guralnick and Kantor in 2000 \cite{GuralnickKantor00}, and Burness, Guralnick and Harper \cite{BurnessGuralnickHarper21} recently classified the finite $\frac{3}{2}$-generated groups.

In this paper, we tie together these two influential themes. 

Let $G \leq \Homeo(\C)$ be a group of homeomorphisms of Cantor space $\C$.
We say that $G$ is \emph{vigorous} if for all clopen subsets $A$ of $\C$ and all proper nonempty clopen subsets $B$ and $C$ of $A$,
there exists $\gamma \in G$ such that $B\gamma \subseteq C$ and $\gamma$ fixes $\C \setminus A$ pointwise.  
We prove that finitely generated simple vigorous groups share many of the strong generation properties of finite simple groups, such as $\frac{3}{2}$-generation mentioned above.
To avoid doubt, we note from the outset that all vigorous groups are infinite and nonabelian (see \cite[Theorem~3.4]{BleakElliottHyde24} for example).

Our first main result concerns generating sets consisting of finite order elements.
Let $m$ and $n$ be positive integers. 
A group $G$ is \emph{$(m, n)$-generated} if there exist $\alpha, \beta \in G$ where $\alpha$ has order $m$ and $\beta$ has order $n$ such that $G = \< \alpha, \beta \>$. 
For an $(m,n)$-generated group $G$, if $m = 1$ or $n = 1$, then $G$ is cyclic, and if $m = n = 2$, then $G$ is a (possibly infinite) dihedral group. 
In particular, if $G$ is an $(m,n)$-generated nonabelian simple group with $m \leq n$, then we necessarily have $m \geq 2$ and $n \geq 3$. 
Our first theorem establishes the converse for finitely generated simple vigorous groups.

\begin{thm} 
\label{thm:torsion}
Let $G \leq \Homeo(\C)$ be a finitely generated simple vigorous group. 
Let $m \geq 2$ and $n \geq 3$. 
Then $G$ is $(m, n)$-generated.
\end{thm}

A group is $(2, 3)$-generated if and only if it is a quotient of the modular group $\mathrm{PSL}_2(\mathbb{Z}) = C_2 \star C_3$, so particular attention has been given to this case. 
For example, Sapir asked whether Thompson's group $V$ is $(2,3)$-generated \cite{Sapir} (see also \cite[Question~6.9]{BleakElliottHyde24}).
Since $V$ is a finitely generated simple vigorous group, Theorem~\ref{thm:torsion} answers this question affirmatively.

\begin{cor}
\label{cor:torsion_v_intro}
Thompson's group $V$ is $(2, 3)$-generated.
\end{cor}

All our proofs are constructive.
In particular, Section~\ref{s:torsion} describes generators that witness Theorem~\ref{thm:torsion}, and the case $(m, n) = (2, 3)$ is illustrated in Figure~\ref{fig:torsion} in that section.
For Thompson's group $V$, we can be even more explicit as the following example shows.

\begin{ex} 
\label{ex:torsion}
In this example, we make use of the permutation notation for torsion elements of $V$ (see Section~\ref{ss:prelims_generation_v}).
By Corollary~\ref{cor:torsion_v}, we have $|\alpha| = 2$, $|\beta| = 3$ and $V = \< \alpha, \beta \>$ for the elements
\begin{gather*}
\alpha = (00 \ 010100)(010110 \ 010111) (100 \ 0110)(110 \ 0111) \cdot (\gamma_{[101]})^{-1} \gamma_{[111]} (101 \ 111) \\
\beta = (01 \ 10 \ 11).
\end{gather*}
Here $\gamma_{[101]}$ and $\gamma_{[111]}$ are the elements of $V$ with support contained in $101\C$ and $111\C$, respectively, that correspond to the element $\gamma = \phi\psi \in V$ with $\phi, \psi \in V$ defined as
\begin{gather*}
\phi = (0000 \ 010 \ 00011 \ 00010)(0010 \ 011 \ 0011) \cdot (000 \ 001) \\
\psi = (1000 \ 10010 \ 10011 \ 110)(1010 \ 1011 \ 111).
\end{gather*}
The authors are not aware of a simpler pair of elements $\alpha, \beta \in V$ that generate $V$ and satisfy $|\alpha| = 2$ and $|\beta| = 3$.
\end{ex}

As already noted, a nonabelian simple group cannot be generated by two involutions, but when we prove Theorem~\ref{thm:torsion}, we actually show that the generator of order $n$ can be chosen to be the product of two involutions, so with $m = 2$ we obtain the following.

\begin{thm} 
\label{thm:involutions}
Every finitely generated simple vigorous group is generated by three involutions.
\end{thm}

The special case of Theorem~\ref{thm:involutions} for Higman--Thompson groups was recently independently proved in \cite{ScheslerSkipperWu} using a different method.  
While they show that those groups can be generated by three involutions, the construction in \cite{ScheslerSkipperWu} does not give $(2, 3)$-generation.

\begin{ex} 
\label{ex:involutions}
For Thompson's group $V$ we can explicitly give three very simple involutions.
As proved in Section~\ref{ss:prelims_generation_v}, we have $|\alpha| = |\beta| = |\gamma| = 2$ and $V = \< \alpha, \beta, \gamma \>$ where
\[
\alpha = (000 \ 001)(010 \ 011)(10 \ 11) \quad \beta = (001 \ 010)(10 \ 11) \quad \gamma = (000 \ 10)(001 \ 110)(010 \ 111).
\]
The tree pairs for these three involutions are given in Figure~\ref{fig:tree_pairs}.
\end{ex}

\begin{figure}
  \[
  \begin{tikzpicture}[
      inner sep=0pt,
      baseline=-30pt,
      level distance=20pt,
      level 1/.style={sibling distance=60pt},
      level 2/.style={sibling distance=30pt},
      level 3/.style={sibling distance=15pt}
    ]
    \node (root) [circle,fill] {}
    child {node (0) [circle,fill] {}
      child {node (00) [circle,fill] {}
        child {node (000) {$1$}}
        child {node (001) {$2$}}}
      child {node (01) [circle,fill] {}
      	child {node (010) {$3$}}
      	child {node (011) {$4$}}}}
    child {node (1) [circle,fill] {}
      child {node (10) {$5$}}
      child {node (11) {$6$}}};
  \end{tikzpicture}
  \ \xrightarrow{\,\alpha\,} \
  \begin{tikzpicture}[
      inner sep=0pt,
      baseline=-30pt,
      level distance=20pt,
      level 1/.style={sibling distance=60pt},
      level 2/.style={sibling distance=30pt},
      level 3/.style={sibling distance=15pt}
    ]
    \node (root) [circle,fill] {}
    child {node (0) [circle,fill] {}
      child {node (00) [circle,fill] {}
        child {node (000) {$2$}}
        child {node (001) {$1$}}}
      child {node (01) [circle,fill] {}
      	child {node (010) {$4$}}
      	child {node (011) {$3$}}}}
    child {node (1) [circle,fill] {}
      child {node (10) {$6$}}
      child {node (11) {$5$}}};
  \end{tikzpicture}
  \]
\vspace{11pt}
  \[
  \begin{tikzpicture}[
      inner sep=0pt,
      baseline=-30pt,
      level distance=20pt,
      level 1/.style={sibling distance=60pt},
      level 2/.style={sibling distance=30pt},
      level 3/.style={sibling distance=15pt}
    ]
    \node (root) [circle,fill] {}
    child {node (0) [circle,fill] {}
      child {node (00) [circle,fill] {}
        child {node (000) {$1$}}
        child {node (001) {$2$}}}
      child {node (01) [circle,fill] {}
      	child {node (010) {$3$}}
      	child {node (011) {$4$}}}}
    child {node (1) [circle,fill] {}
      child {node (10) {$5$}}
      child {node (11) {$6$}}};
  \end{tikzpicture}
  \ \xrightarrow{\,\beta\,} \
  \begin{tikzpicture}[
      inner sep=0pt,
      baseline=-30pt,
      level distance=20pt,
      level 1/.style={sibling distance=60pt},
      level 2/.style={sibling distance=30pt},
      level 3/.style={sibling distance=15pt}
    ]
    \node (root) [circle,fill] {}
    child {node (0) [circle,fill] {}
      child {node (00) [circle,fill] {}
        child {node (000) {$1$}}
        child {node (001) {$3$}}}
      child {node (01) [circle,fill] {}
      	child {node (010) {$2$}}
      	child {node (011) {$4$}}}}
    child {node (1) [circle,fill] {}
      child {node (10) {$6$}}
      child {node (11) {$5$}}};
  \end{tikzpicture}
  \]
\vspace{11pt}
  \[
  \begin{tikzpicture}[
      inner sep=0pt,
      baseline=-30pt,
      level distance=20pt,
      level 1/.style={sibling distance=60pt},
      level 2/.style={sibling distance=30pt},
      level 3/.style={sibling distance=15pt}
    ]
    \node (root) [circle,fill] {}
    child {node (0) [circle,fill] {}
      child {node (00) [circle,fill] {}
        child {node (000) {$1$}}
        child {node (001) {$2$}}}
      child {node (01) [circle,fill] {}
      	child {node (010) {$3$}}
      	child {node (011) {$4$}}}}
    child {node (1) [circle,fill] {}
      child {node (10) {$5$}}
      child {node (11) [circle,fill] {}
      	child {node (110) {$6$}}
      	child {node (111) {$7$}}}};
  \end{tikzpicture}
  \ \xrightarrow{\,\gamma\,} \
  \begin{tikzpicture}[
      inner sep=0pt,
      baseline=-30pt,
      level distance=20pt,
      level 1/.style={sibling distance=60pt},
      level 2/.style={sibling distance=30pt},
      level 3/.style={sibling distance=15pt}
    ]
    \node (root) [circle,fill] {}
    child {node (0) [circle,fill] {}
      child {node (00) [circle,fill] {}
        child {node (000) {$5$}}
        child {node (001) {$6$}}}
      child {node (01) [circle,fill] {}
      	child {node (010) {$7$}}
      	child {node (011) {$4$}}}}
    child {node (1) [circle,fill] {}
      child {node (10) {$1$}}
      child {node (11) [circle,fill] {}
      	child {node (110) {$2$}}
      	child {node (111) {$3$}}}};
  \end{tikzpicture}
  \]
\caption{The generating involutions $\alpha, \beta, \gamma \in V$ given in Example~\ref{ex:involutions}.}
\label{fig:tree_pairs}
\end{figure}
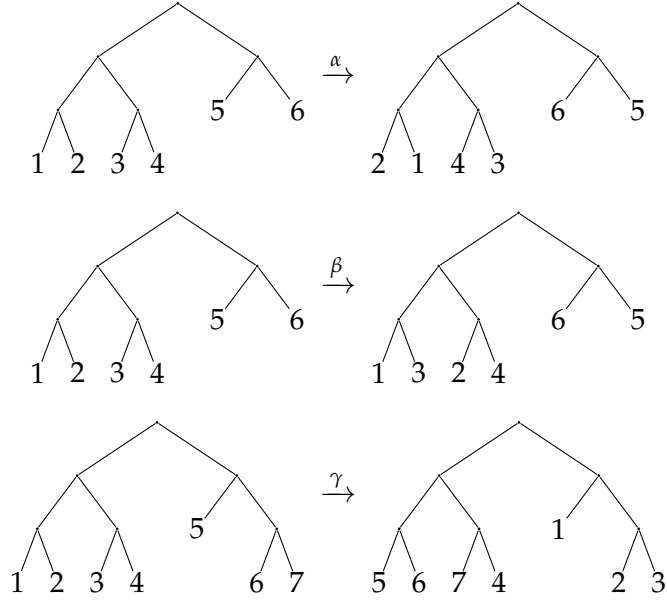

\begin{rem}
\label{rem:torsion_finite}
Let us compare Theorems~\ref{thm:torsion} and~\ref{thm:involutions} with what is known for finite simple groups. 
Combining several results \cite{LiebeckShalev96, LubeckMalle99, Miller01}, every sufficiently large nonabelian finite simple group is $(2, 3)$-generated except the symplectic groups $\mathrm{PSp}_4(2^f)$ and $\mathrm{PSp}_4(3^f)$ and the Suzuki groups ${}^2B_2(2^f)$, which are all $(2, 5)$-generated, and a conjecture of Conder, which has seen recent attention, asserts that every nonabelian finite simple group is, in fact, $(2, 3)$-, $(2, 5)$-, or $(2,7 )$-generated (see, for example, \cite[Conjecture~2.9]{Burness19}). 
Every finite simple group other than the unitary group $\mathrm{PSU}_3(3)$ is also known to be generated by three involutions \cite{MalleSaxlWeigel94}. 
\end{rem}

Combining Theorems~\ref{thm:torsion} and~\ref{thm:involutions}, a finitely generated simple vigorous group can be generated by elements of orders $m_1, \dots, m_k$, all at least two, provided that $k \geq 2$ and $\{ m_1, m_2 \} \neq \{ 2 \}$ when $k = 2$.
However, we can say more.
A generating set $X$ of a group $G$ is \emph{minimal} if no proper subset of $X$ generates $G$. 
If $G$ is finitely generated, then every minimal generating set of $G$ is finite, and we write $d(G)$ and $m(G)$ for the infimum and supremum of the sizes of minimal generating sets of $G$. 
By Tarski's Irredundant Basis Theorem \cite{Tarski75}, for every integer $k$ satisfying $d(G) \leq k \leq m(G)$ there is a minimal generating set of $G$ of size $k$. 
For finite simple groups $G$, it is well known that $d(G) \leq 2$, but $m(G)$ is much harder to analyse (e.g. Whiston's proof that $m(A_n) = n - 2$ \cite{Whiston00} uses the Classification of Finite Simple Groups, see \cite{Harper23} for further information). 
In contrast, for finitely generated simple vigorous groups $G$, we know $m(G)$.
Indeed, we have the following, which subsumes Theorems~\ref{thm:torsion} and~\ref{thm:involutions}.

\begin{thm}
\label{thm:minimal}
Let $G \leq \Homeo(\C)$ be a finitely generated simple vigorous group. 
Let $k \geq 2$ and let $m_1, \dots, m_k \geq 2$ be integers where $\{ m_1, m_2 \} \neq \{ 2 \}$ if $k = 2$.
Then $G$ has a minimal generating set $\{ \alpha_1, \dots, \alpha_k \}$ of size $k$ where $|\alpha_i| = m_i$ for all $1 \leq i \leq k$.
In particular, $m(G) = \infty$.
\end{thm}

\begin{ex}
\label{ex:minimal}
For Thompson's group $V$, there is a simple construction to see that $m(V) = \infty$. 
Let $n \geq 2$ and let $x_1, \dots, x_N$ be the words of length $n$ over $\{0,1\}$. 
As proved in Section~\ref{ss:prelims_generation_v},
\[
\{ (x_1 \ x_2), \dots, (x_1 \ x_N), (x_1 \ x_N0 \ x_N1) \}
\]
is a minimal generating set of $V$ of size $N = 2^n$. 
This proves that $V$ has minimal generating sets of arbitrarily large size. 
Tarski's Irredundant Basis Theorem, therefore, implies that $V$ has a minimal generating set of size $k$ for every integer $k \geq 2$.
\end{ex}

Theorems~\ref{thm:torsion} and~\ref{thm:involutions} have two applications.
First, in 1978, Wiegold and Wilson \cite{WiegoldWilson78} proved that if $G$ is a finitely generated infinite simple group, then $d(G) \leq d(G^k) \leq d(G) + 1$ for all positive integers $k$.
Therefore, if $G \leq \Homeo(\C)$ is a finitely generated simple vigorous group, then $d(G^k) \in \{ 2, 3 \}$ for all positive integers $k$, and, as a consequence of Theorem~\ref{thm:torsion}, we actually have $d(G^k) = 2$. 
In fact, our methods give the following much stronger result.

\begin{thm}
\label{thm:products}
Let $G_1, \dots, G_l \leq \Homeo(\C)$ be finitely generated simple vigorous groups.
Let $k \geq 2$ and let $m_1, \dots, m_k \geq 2$ be integers where $\{ m_1, m_2 \} \neq \{ 2 \}$ if $k = 2$.
Then $G_1 \times \cdots \times G_l$ has a minimal generating set $\{ \alpha_1, \dots, \alpha_k \}$ of size $k$ where $|\alpha_i| = m_i$ for all $1 \leq i \leq k$.
\end{thm}

Second, in 1974, Hall \cite{Hall74} proved that every countable group embeds in a finitely generated simple group, and Schupp \cite{Schupp76} showed that the simple group can be chosen to be $(m, n)$-generated for any $m \geq 2$ and $n \geq 3$ (see also \cite{Goryushkin74}).
In 1998, Bridson \cite{Bridson98} proved that every finitely generated group quasi-isometrically embeds in a finitely generated group with no proper finite index subgroup.
Belk and Zaremsky \cite{BelkZaremsky22} recently proved that every finitely generated group quasi-isometrically embeds into a finitely generated twisted Brin--Thompson group, which, as they note, is a simple vigorous group. 
Therefore, Theorem~\ref{thm:torsion} implies the following, which simultaneously generalises the results of Schupp and Bridson for finitely generated groups.

\begin{cor}
\label{cor:embedding}
Let $m \geq 2$ and $n \geq 3$. 
Every finitely generated group quasi-isometrically embeds in an $(m, n)$-generated simple group.
\end{cor}

The Boone--Higman conjecture \cite{BooneHigman74} asserts that a finitely generated group has solvable word problem if and only if it embeds in a finitely presented simple group.
The Permutational Boone--Higman conjecture \cite{Zaremsky24}, a recent variant on this long-standing conjecture, would imply that this simple group could be taken to be a vigorous group.

Our next main result establishes a strong version of $\frac{3}{2}$-generation mentioned earlier.

\begin{thm} 
\label{thm:spread}
Let $G \leq \Homeo(\C)$ be a finitely generated simple vigorous group. 
Then there exists $\sigma \in G$ of order $30$ such that for any nontrivial $\alpha \in G$ there exists $\beta \in \sigma^G$ such that $\< \alpha, \beta \> = G$. 
\end{thm}

Donoven and Harper \cite{DonovenHarper20} proved that Thompson's group $V$, as well as the Higman--Thompson groups $V_n'$ and the Brin--Thompson groups $nV$, are $\frac{3}{2}$-generated, giving the first examples of infinite $\frac{3}{2}$-generated groups (other than the infinite cyclic group and Tarski monsters). 
Theorem~\ref{thm:spread} vastly generalises this result both by giving a much stronger conclusion and by uniformly handling a much broader class of groups.

Uniform spread is a generalisation of $\frac{3}{2}$-generation that has seen much attention for finite groups in the last couple of decades (see Harper's survey \cite{Harper24Survey}). 
The \emph{uniform spread} of $G$, written $u(G)$, is the supremum over integers $k$ for which there exists $\sigma \in G$ such that for any nontrivial $\alpha_1, \dots, \alpha_k \in G$, there exists $\beta \in \sigma^G$ such that $\< \alpha_1, \beta \> = \dots = \< \alpha_k, \beta \> = G$. 
If $u(G) \geq 1$, then certainly $G$ is $\frac{3}{2}$-generated, and, in fact, the symmetric group $S_6$ is the only nonabelian finite group $G$ that is $\frac{3}{2}$-generated but where $u(G) = 0$ (see \cite[Theorem~3]{BurnessGuralnickHarper21}).

Theorem~\ref{thm:spread} has the following immediate consequence.

\begin{cor}
\label{cor:spread}
Let $G \leq \Homeo(\C)$ be a finitely generated simple vigorous group. 
Then $u(G) \geq 1$.
\end{cor}

It is natural to ask for the precise value of the uniform spread of simple vigorous groups and, in particular, Thompson's group $V$. 
We pose the following question.

\begin{que}
\label{que:spread}
Is $u(G) = \infty$ for all finitely generated simple vigorous groups $G \leq \Homeo(\C)$?
\end{que}

The authors are currently considering this question in ongoing work, but we make some comments in Section~\ref{s:spread}.
In particular, Question~\ref{que:spread_refined} gives a more fine-grained version of Question~\ref{que:spread}.
Very little is known in this direction, and even the following special case was open.
However, we settle this case, using Corollary~\ref{cor:torsion_v_intro} as a key ingredient.

\begin{thm}
\label{thm:spread_two_swaps}
There exists an element $\sigma \in V$ such that $\< (00 \ 01), \sigma \> = \< (10 \ 11), \sigma \> = V$.
\end{thm}

Let us comment on the structure of the paper.
Section~\ref{s:prelims} gives various preliminary results on vigorous groups. 
This includes Theorem~\ref{thm:criterion}, which is a new generation criterion for $V$ that we expect to have further applications.
More specifically, many important questions concerning Thompson's group $V$ require showing that a particular set of elements generate $V$ and many ad hoc arguments have been provided in the literature to do this.
Theorem~\ref{thm:criterion} is the most general result that we are aware of for showing that a set of elements generate $V$, so we expect that this will simplify many future proofs in this field.

Turning to our main results, Section~\ref{ss:torsion_construction} gives a general construction for generating sets of simple vigorous groups, and Proposition~\ref{prop:technical} is the main generation result in the paper. 
As with Theorem~\ref{thm:criterion}, we have stated this result in such generality to allow readers to obtain further generation results as an immediate consequence.
We use this construction to prove Theorem~\ref{thm:products} (and its special cases Theorems~\ref{thm:torsion}, \ref{thm:involutions} and~\ref{thm:minimal}) in Section~\ref{ss:torsion_applications}.
Theorems~\ref{thm:spread} and~\ref{thm:spread_two_swaps} are proved in Section~\ref{ss:spread_lower}, and Section~\ref{ss:spread_upper} features further comments on uniform spread.

Actions are on the right, so, in particular, $x^y$ denotes $y^{-1}xy$ and $[x,y]$ denotes $x^{-1}y^{-1}xy$.

\section{Preliminaries} 
\label{s:prelims}

\subsection{Properties of simple vigorous groups} 
\label{ss:prelims_properties}

This section recalls some standard properties of simple vigorous groups, following \cite{BleakElliottHyde24}.
Let $\C$ be Cantor space.
As usual, for $U \subseteq \C$ and $G \leq \Homeo(\C)$, define the \emph{pointwise stabiliser} $G_{(U)} = \{ \gamma \in G \mid \text{$u\gamma = u$ for all $u \in U$} \}$ and the \emph{rigid stabiliser} $G_U = G_{(\C \setminus U)}$. 

A group $G \leq \Homeo(\C)$ is \emph{vigorous} over $U \subseteq \C$ if for any clopen $\emptyset \subsetneq A \subseteq U$ and any clopen $\emptyset \subsetneq B,C \subsetneq A$ there exists $\gamma \in G_A$ such that $B\gamma \subseteq C$; if $U = \C$, then we say that $G$ is \emph{vigorous}. 
Let $\mathcal{K}^{\text{(fg)}}$ be the class of (finitely generated) simple $G \leq \Homeo(\C)$ that are vigorous, and let $\mathcal{K}_U^{\text{(fg)}}$ be the class of (finitely generated) simple $G \leq \Homeo(\C)_U$ that are vigorous over $U$. 
The following is  \cite[Proposition~5.4]{BleakElliottHyde24}.

\begin{lemma}
\label{lem:restriction}
Let $G \leq \Homeo(\C)$, and let $\emptyset \subsetneq U \subsetneq \C$ be clopen. 
If $G \in \mathcal{K}$, then $G_U \in \mathcal{K}_U$, and if $G \in \mathcal{K}^{\text{fg}}$, then $G_U \in \mathcal{K}^{\text{fg}}_U$.
\end{lemma}

By \cite[Theorem~4.18]{BleakElliottHyde24}, if $G \leq \Homeo(\C)$ is vigorous and simple, then it is generated by elements of \emph{small support} and \emph{approximately full}. 
Let us recall the definitions of these two concepts. 
First, an element $\gamma$ of $G \leq \Homeo(\C)$ has \emph{small support} if $\gamma \in G_U$ for a clopen set $U \subsetneq \C$. 
Second, a group $G \leq \Homeo(\C)$ is said to be \emph{approximately full} if for any $\gamma_1, \dots, \gamma_k \in G$ and any clopen sets $U_1, \dots, U_k$ such that $\C = \bigsqcup_{i=1}^{k}U_i = \bigsqcup_{i=1}^{k} U_i\gamma_i$, for all $1 \leq j \leq k$ there exists $\gamma \in G$ such that $\gamma|_{U_i} = \gamma_i|_{U_i}$ for all $i$ with $i \neq j$.

The following result, which is \cite[Lemma~5.2]{BleakElliottHyde24}, is the main way in this paper that we will exploit the property of being generated by elements of small support.

\begin{lemma}
\label{lem:union}
Let $\emptyset \subsetneq U, V \subsetneq \C$ be clopen and $U \cap V \neq \emptyset$. 
Let $G \leq \Homeo(\C)_{U \cup V}$. 
Assume that $G$ is vigorous and generated by elements of small support. 
Then $G =  \< G_U, G_V \>$.
\end{lemma}

\subsection{Homology groups of clopen sets} 
\label{ss:prelims_homology}

For this section, let $G \leq \Homeo(\C)$ be a vigorous approximately full group. 
Following the description in \cite[Section~5.2.1]{BleakElliottHyde24}, itself following other work including of Matui \cite{Matui12}, we define an abelian group that allows us to study the orbits of clopen subsets of $\C$ under the action of $G$ and how these orbits interact. 

For a clopen set $\emptyset \subsetneq U \subsetneq \C$, write $[U]_G = \{ U\gamma \mid \gamma \in G \}$. Now we may define
\[
\X_G = \{ [U]_G \mid \text{$\emptyset \subsetneq U \subsetneq \C$ clopen} \}.
\] 
For clopen sets $\emptyset \subsetneq U, V \subsetneq \C$, define $[U]_G + [V]_G$ as $[U\mu \cup V\nu]_G$ for elements $\mu, \nu \in G$ such that $U\mu \cap V\nu = \emptyset$ and $U\mu \cup V\nu \neq \C$. 
Finally, write 
\begin{gather*}
[\emptyset]_G = \{ U\gamma \setminus U \mid \text{$\emptyset \subsetneq U \subsetneq \C$ clopen, $\gamma \in G$, $U \subsetneq U\gamma$} \} \\
[\C]_G = \{ \C \setminus U \mid U \in [\emptyset]_G \}.
\end{gather*} 

The following result, which establishes the fundamental properties of $\X_G$, combines \cite[Proposition~5.8 \& Lemma~5.9]{BleakElliottHyde24}.

\begin{lemma}
\label{lem:homology}
Let $G$ be a vigorous approximately full group. 
Then the following hold:
\begin{enumerate}
\item $\X_G$ is a well-defined abelian group under addition
\item the identity of $\X_G$ is $[\emptyset]_G$
\item the inverse of $[U]_G$ is $[U\mu \setminus (U \cup U\nu)]_G$ for $\mu, \nu \in G$ such that $U \subsetneq U\mu$ and $U\nu \subsetneq U\mu \setminus U$
\item $[U]_G + [\C \setminus U]_G = [\C]_G$.
\end{enumerate}
\end{lemma}

The next result, which is \cite[Lemma~5.13]{BleakElliottHyde24}, highlights the utility of the group $\X_G$.

\begin{lemma}
\label{lem:partition}
Let $G$ be a vigorous approximately full group. 
Let $X_1, \dots, X_k \in \X_G$. 
There exist $U_1 \in X_1$, \dots, $U_k \in X_k$ such that $\mathfrak{C} = \bigsqcup_{i=1}^{k} U_i$ if and only if $\sum_{i=1}^k X_i = [\C]_G$.
\end{lemma}

We conclude this section by introducing some terminology that we will use in the proofs of our main theorems.

Let $\emptyset \subsetneq U \subsetneq \C$ be clopen.
Let $\gamma_1, \dots, \gamma_k \in G$ such that $U\gamma_i \cap U\gamma_j \neq \emptyset$ if and only if $i=j$. 
For $g \in \Sym(k)$, the element of $\Homeo(\C)$ that we say \emph{enacts the permutation} $U\gamma_i \mapsto U\gamma_{ig}$ is 
\[
\pi \in \Homeo(\C)_{\bigsqcup_{i=1}^k U\gamma_i} \text{ defined as } \pi|_{U\gamma_i} = \gamma_i^{-1}\gamma_{jg}|_{U\gamma_i} \text{for $1 \leq i \leq k$},
\]
which we denote by the permutation $U\gamma_i \mapsto U\gamma_{ig}$ written in cycle notation.
Beware that this notation depends on the description of each set as $U\gamma_i$ and not just on the set $U\gamma_i$ itself.

For the following lemma, which is \cite[Lemma~5.12]{BleakElliottHyde24}, we say that $X \in \X_G$ is \emph{even} if there exists $Y \in \X_G$ such that $X = 2Y$.

\begin{lemma}
\label{lem:existence}
Let $G$ be a vigorous approximately full group. 
Let $\emptyset \subsetneq U \subsetneq \C$ be clopen.
Let $\gamma_1, \dots, \gamma_k \in G$ such that $U\gamma_i \cap U\gamma_j \neq \emptyset$ if and only if $i=j$. 
Let $g \in \Sym(k)$. 
If $g \in \Sym(k)$ is an even permutation or $[U]_G \in \X_G$ is even, then the element enacting the permutation $U\gamma_i \mapsto U\gamma_{ig}$ is contained in $G$.
\end{lemma}

The following straightforward conjugation property is useful.

\begin{lemma}
\label{lem:conjugation}
Let $G$ be a vigorous approximately full group. 
Let $\emptyset \subsetneq U \subsetneq \C$ be clopen.
Let $\gamma_1, \dots, \gamma_k \in G$ such that $U\gamma_i \cap U\gamma_j \neq \emptyset$ if and only if $i=j$.
Let $\pi = (U\gamma_1 \ \dots \ U\gamma_k)$ and let $\gamma \in G$.
Then 
\[
\pi^\gamma = (U\gamma_1\gamma \ \dots \ U\gamma_k\gamma).
\]
\end{lemma}

\begin{proof}
Write $g = (1 \ \cdots k) \in \Sym(k)$.
Note that 
\[
\Supp(\pi^\gamma) = \Supp(\pi)\gamma \subseteq (U\gamma_1 \sqcup \dots \sqcup U\gamma_k)\gamma = U\gamma_1\gamma \sqcup \dots \sqcup U\gamma_k\gamma.
\]
Moreover, for all $1 \leq i \leq k$, we have 
\[
\pi^\gamma|_{U\gamma_i\gamma} = \gamma^{-1}|_{U\gamma_i\gamma}(\pi\gamma)|_{U\gamma_i} = \gamma^{-1}|_{U\gamma_i\gamma}(\gamma_i^{-1}\gamma_{ig}\gamma)|_{U\gamma_i} = ((\gamma_i\gamma)^{-1}(\gamma_{ig}\gamma))|_{U\gamma_i\gamma}. \qedhere
\] 
\end{proof}

We conclude with an analogue of the fact that the alternating group $A_n$ is generated by the set of three-cycles.

\begin{lemma}
\label{lem:three-cycles}
Let $G \leq \Homeo(\C)$ be a simple vigorous group.
Let $\emptyset \subsetneq U \subsetneq \C$ be clopen and let
\[
X = \{ (U\gamma_1 \ U\gamma_2 \ U\gamma_3) \mid \text{$\gamma_1, \gamma_2, \gamma_3 \in G$ and $U\gamma_1 \sqcup U\gamma_2 \sqcup U\gamma_3 \subsetneq \C$}\}.
\]
Then $G = \< X \>$.
\end{lemma}

\begin{proof}
Since $G$ is vigorous, there certainly exist $\gamma_1, \gamma_2, \gamma_3 \in G$ such that $U\gamma_1 \sqcup U\gamma_2 \sqcup U\gamma_3 \subsetneq \C$, so $X$ contains nontrivial elements. 
Now let $\gamma_1, \gamma_2, \gamma_3 \in G$ such that $U\gamma_1 \sqcup U\gamma_2 \sqcup U\gamma_3 \subsetneq \C$ and let $\gamma \in G$.
Then, by Lemma~\ref{lem:conjugation},
\[
(U\gamma_1 \ U\gamma_2 \ U\gamma_3)^\gamma = (U\gamma_1\gamma \ U\gamma_2\gamma \ U\gamma_3\gamma) \in X,
\]
so $X$ is a normal subset of $G$.
The simplicity of $G$ now gives $G = \< X \>$.
\end{proof}

\subsection{Generating pairs for simple vigorous groups}
\label{ss:prelims_torsion}

This section is dedicated to stating a version of \cite[Theorem~5.15]{BleakElliottHyde24}, which includes further detail that follows from the proof given in \cite{BleakElliottHyde24}. We outline the proof.

\begin{proposition}
\label{prop:torsion}
Let $G \leq \Homeo(\C)$ be a finitely generated simple vigorous group. 
Let $m, n, r \geq 2$ be integers. 
Then there exist $\rho, \sigma \in G$ such that the following hold
\begin{enumerate}
\item $|\rho| = m$
\item $|\sigma| \geq n$
\item $\rho$ has small support
\item $G = \< \rho, \sigma \>$
\item $[\sigma^k,(\rho^\epsilon\sigma)^k] \neq 1$ for all $k \leq r$ and $\epsilon \in \{ +1, -1 \}$.
\end{enumerate} 
Moreover, there exists an integer $N \geq n$ (depending on $m$, $n$, $r$ and $G$) such that for each integer $t \geq N$ we can chose the element $\sigma$ to have order $t(2t+1)$. 
\end{proposition}

\enlargethispage{3pt}
\begin{proof}
Fix a clopen set $\emptyset \subsetneq C \subsetneq \C$ such that $[C]_G = [\C]_G$ and a clopen set $\emptyset \subsetneq D \subsetneq \C \setminus C$ such that $[D]_G = -2 [C]_G$. 
By Lemma~\ref{lem:restriction}, $G_{C \cup D} \in \mathcal{K}_{C \cup D}^{\text{fg}}$ since $G \in \mathcal{K}^{\text{fg}}$.
Let 
\[
P = \{ \alpha \in G_{C \cup D} \mid |\alpha| = m \}.
\]
Note $P \neq \emptyset$ since $\Alt(m + 2)$ embeds in $G_{C \cup D}$ by Lemma~\ref{lem:existence}, and $G_{C \cup D} = \< P \>$ since $P$ is a nontrivial normal subset of $G_{C \cup D}$. 
Let 
\[
Q = \{ [\alpha, \beta] \mid \alpha, \beta \in P \}.
\] 
Note $Q \neq \{ 1 \}$ since $P$ is a generating set for a nonabelian group, and $G_{C \cup D} = \< Q \>$ since $Q$ is a nontrivial normal subset of $G_{C \cup D}$. 
Since $G_{C \cup D}$ is finitely generated, a finite subset of $Q$ must generate $G_{C \cup D}$. 
Therefore, we can fix $\psi_1, \dots, \psi_j, \omega_1, \dots, \omega_j \in G_{C \cup D}$ of order $m$ such that $G_{C \cup D} = \< [\psi_1,\omega_1], \dots, [\psi_j,\omega_j] \>$. 

Write $N = \max{\{rj(j+1), n}\}$ and fix $t \geq N$.
Since $(2t+1)[C]_G + t[D]_G = [C]_G = [\C]_G$, Lemma~\ref{lem:partition} implies that $\C = \bigsqcup_{A \in X} A \sqcup \bigsqcup_{B \in Y} B$ where $X \subseteq [C]_G$ with $C \in X$ and $|X| = 2t+1$ and $Y \subseteq [D]_G$ with $D \in Y$ and $|Y| = t$. 
By Lemma~\ref{lem:existence}, since an $(2t+1)$-cycle is even, fix $\sigma_1 \in G_{\bigsqcup_{A \in X}A}$ enacting an $(2t+1)$-cycle on $X$, and since $[D]_G$ is even, fix $\sigma_2 \in G_{\bigsqcup_{B \in Y}B}$ enacting an $t$-cycle on $Y$.
Let $\sigma = \sigma_1 \sigma_2$, noting that $|\sigma| = \mathrm{lcm}(2t+1,t) = (2t+1)t \geq n$, so (ii) holds, as does the stronger claim in the final sentence of the statement.

Note $\C = \bigcup_{k=1}^{(2t+1)t} (C \cup D)\sigma^k$ and $(C \cup D)\sigma^k \cap (C \cup D)\sigma^l = \emptyset$ for $1 \leq k < l \leq t$. 
Let $\Gamma$ be the graph with vertex set $\{ (C \cup D)\sigma^k \mid k \in \Z \}$ and where $u$ and $v$ are adjacent if and only if $u \cap v \neq \emptyset$. 
Note $\Gamma$ is finite as $|\sigma| < \infty$ and $\Gamma$ is connected as $|\sigma_1| = 2t+1$ and $|\sigma_2| = t$ are coprime.
Therefore, by repeated application of Lemma~\ref{lem:union}, we see that $G = \< G_{C \cup D}, \sigma \>$.

Let $\rho = \prod_{i=1}^{j} \psi_i^{\sigma^i}\omega_i^{\sigma^{i(j+1)}}$. 
The elements $\psi_1^{\sigma^1}, \dots, \psi_j^{\sigma^j}, \omega_1^{1(j+1)}, \dots, \omega_j^{j(j+1)}$ have pairwise disjoint support, so, in particular, $|\rho| = m$, so (i) holds.
Using \cite[Lemma~5.14]{BleakElliottHyde24}, $[\rho^{\sigma^{-k}},\rho^{\sigma^{-k(j+1)}}] = [\psi_k,\omega_k]$ for all $1 \leq k \leq j$. 
Therefore, (iv) holds since 
\[
G = \< G_{C \cup D}, \sigma \> = \< [\psi_1,\omega_1], \dots, [\psi_j,\omega_j], \sigma \> = \< \sigma, \rho \>.
\]

Recall that $\C = \bigsqcup_{i=1}^{t}(C \cup D)^{\sigma^i} \sqcup \bigsqcup_{i=t+1}^{2t+1} C^{\sigma^i}$ and $t \geq rj(j+1)$. 
Note for $1 \leq i \leq t$, the element $\rho$ acts nontrivially on $(C \cup D)^{\sigma^i}$ if and only if $i \in [1,j] \cup [1,j](j+1) \subseteq [1,j(j+1)]$.
In particular, $\rho$ has small support, so (iii) holds.
Moreover, for all $1 \leq k \leq r$ and $\epsilon \in \{ +, - \}$, the element $\chi = \rho^\epsilon\rho^{\epsilon\sigma^{-1}} \dots \rho^{\epsilon\sigma^{-(k-1)}}$ acts nontrivially on $(C \cup D)\sigma^{j(j+1)}$ but $\chi^{\sigma^{-k}}$ does not, so $\chi \neq \chi^{\sigma^{-k}}$.
Therefore, (v) holds since
\vspace{-2pt}
\[
[\sigma^k, (\rho^\epsilon\sigma)^k] = [\sigma^k,\chi\sigma^k] = [\sigma^k,\chi]^{\sigma^k} \neq 1. \qedhere
\]
\end{proof}

\subsection{\boldmath Generation criterion for \texorpdfstring{$V$}{V}} 
\label{ss:prelims_generation_v}

We now focus on Thompson's group $V$, first introduced in \cite{ThompsonHandwritten}.
Before stating a generation criterion for $V$, let us recall some convenient notation for describing elements of $V$.

For $u = u_1u_2 \dots u_k \in \{ 0, 1 \}^*$, we write $|u| = k$ and $u\C= \{ uw \mid w \in \C \}$. 
We say that $u, v \in \{ 0, 1 \}^*$ are \emph{incomparable}, and write $u \perp v$, if $u\C \cap v\C = \emptyset$, and we say that a finite set $A \subseteq \{0,1\}^*$ is a \emph{basis} of $\C$ if $\{ u\C \mid u \in A \}$ is a partition of $\C$. 

Thompson's group $V$ is the group of homeomorphisms $\gamma \in \Homeo(\C)$ for which there exists a \emph{basis pair}, namely a bijection $g\: A \to B$ between two bases $A$ and $B$ of $\C$ such that $(uw)\gamma = (ug)w$ for all $u \in A$ and $w \in \C$. 
For $\gamma \in V$ and $u \in \{ 0, 1 \}^*$, we write $\gamma_{[u]}$ for the element of $V$ defined, for each $w \in \C$, as follows
\begin{equation}
\label{eq:restriction}
w \gamma_{[u]} = \left\{
\begin{array}{ll}
u (v \gamma)   & \text{if $w=uv$} \\
w              & \text{if $w \not\in u\C$.} \\
\end{array}
\right.
\end{equation}
In our notation for rigid stabilisers, 
$
V_{u\C} = \{ \gamma_{[u]} \mid \gamma \in V \}
$.
A basis can naturally be identified with the set of leaves of a binary rooted tree, giving the familiar \emph{tree pairs}, as in Figure~\ref{fig:tree_pairs}.

An element $\gamma \in V$ has finite order if and only if it admits a basis pair $(A,g,A)$ (see \cite[Lemma~17]{Scott92}, for example). 
Following Bleak and Quick \cite{BleakQuick17}, in this case, we denote $\gamma \in G$ by the permutation $g \in \mathrm{Sym}(A)$. 
(This is closely related to the permutation notation in Section~\ref{ss:prelims_homology}.)
For example, the elements $\alpha, \beta, \gamma \in V$ defined by the tree pairs in Figure~\ref{fig:tree_pairs} all have order two and they are expressed in permutation notation in Example~\ref{ex:involutions}.

\begin{theorem} 
\label{thm:criterion}
Let $A = \{ u_1, \dots, u_n \}$ be a basis of $\C$ of size $n \geq 3$. Let $B = \{ u_1, \dots, u_{n-1}, u_n0, u_n1 \}$. Then $V = \< \Alt(A), \Alt(B) \>$.
\end{theorem}

\begin{proof}
First assume that $n = 3$ and, for exposition, write $A = \{ u, v, w \}$ and $B = \{ u, v, w0, w1 \}$. 
Since it is impossible to have $|u| = |v| = 1$, without loss of generality, assume that $|u| \geq 2$.
Write $H = \< \Alt(A), \Alt(B) \>$.
Observe that $(v \ w0 \ w1)^{(u \ w \ v)} = (u \ v0 \ v1)$, so 
\[
\Alt(\{ u, v0, v1, w0, w1 \}) = \< (u \ v0 \ v1), (u \ w0 \ w1) \> \leq H.
\] 
In particular, $(v \ w) = (v0 \ w0)(v1 \ w1) \in H$.

Let $x \in \{ 0, 1 \}^*$ such that $u \perp x$ and $|x| \geq |u|$.
By \cite[Lemma~3.4(i)]{DonovenHarper20}, it suffices to prove that $(u \ x) \in H$.
Since $\{ u, v, w \}$ is a basis, there exists $y \in \{ 0, 1 \}^*$ such that $x = vy$ or $x = wy$.
If $(u \ wy) \in H$, then 
\[
(u \ vy) = (u \ wy)^{(v \ w)} \in H,
\]
so without loss of generality, assume that $x = wy$.
We now proceed by induction on $|y|$.
If $|y| = 0$, then $x = w$ and certainly $(u \ w) \in H$.
Now assume that $|y| > 0$ and write $y = az$ where $a \in \{0, 1\}$ and $z \in \{0, 1\}^*$.
By induction, $(u \ wz) \in H$, so 
\begin{gather*}
(u \ w0z) = (u \ wz)^{(v \ w)(v \ w0 \ w1)} \in H \\
(u \ w1z) = (u \ wz)^{(v \ w)(v \ w1 \ w0)} \in H,
\end{gather*}
which proves that $(u \ x) = (u \ wy) = (u \ waz) \in H$, completing the induction.

Now assume that $n > 3$.
For each $1 \leq i \leq n - 2$, let $X_i = u_i\C \sqcup u_{n-1}\C \sqcup u_n\C$, and write $A_i = \{ u_i, u_{n-1}, u_n \}$ and $B_i = \{ u_i, u_{n-1}, u_n0, u_n1 \}$.
Since we have proved the theorem when $n = 3$, we know that $\< \Alt(A_i), \Alt(B_i) \> = V_{X_i}$ for all $1 \leq i \leq n-2$.
Therefore, 
\[
\< \Alt(A), \Alt(B) \> \geq \< \Alt(A_1), \Alt(B_1), \dots, \Alt(A_{n-2}), \Alt(B_{n-2}) \> \geq \< V_{X_1}, \dots, V_{X_{n-2}} \>.
\]
However, $\C = \bigcup_{i=1}^{n-2} X_i$, and for $1 \leq i, j \leq n - 2$, the intersection $X_i \cap X_j$ is nonempty as $u_n\C \subseteq X_i \cap X_j$.
Therefore, repeatedly applying Lemma~\ref{lem:union} yields $\< V_{X_1}, \dots, V_{X_{n-2}} \> = V$, which completes the proof.
\end{proof}

We use Theorem~\ref{thm:criterion} to verify the generating sets given in Examples~\ref{ex:involutions} and~\ref{ex:minimal}.

\begin{proof}[Proof of Example~\ref{ex:involutions}]
Define the elements
\[
\alpha = (000 \ 001)(010 \ 011)(10 \ 11) \quad \beta = (001 \ 010)(10 \ 11) \quad \gamma = (000 \ 10)(001 \ 110)(010 \ 111).
\]
Fix the two bases
\[
A = \{ 000, 001, 010, 011, 10, 11 \} \quad B = \{ 000, 001, 010, 011, 10, 110, 111 \}.
\] 
Since $\alpha\beta = (000 \ 010 \ 011 \ 001)$ and $\gamma = (000 \ 10)(001 \ 110)(010 \ 111)$, we see that 
\[
\Alt(B) \leq \Sym(B) = \< (000 \ 010 \ 011 \ 001), (000 \ 10)(001 \ 110)(010 \ 111) \> \leq \< \alpha, \beta, \gamma \>.
\]
Thus, $(001 \ 010), (000 \ 001 \ 010 \ 011 \ 10) \in \< \alpha\beta, \gamma \>$, so $(10 \ 11) = \beta (001 \ 010) \in \< \alpha, \beta, \gamma \>$.
Hence,
\[
\Alt(A) \leq \Sym(A) = \< (000 \ 001 \ 010 \ 011 \ 10), (10 \ 11) \> \leq \< \alpha, \beta, \gamma \>.
\]
By Theorem~\ref{thm:criterion}, we conclude that $V =  \< \alpha, \beta, \gamma \>$.
\end{proof}

\begin{proof}[Proof of Example~\ref{ex:minimal}]
Let $n \geq 2$, let $x_1, \dots, x_N$ be the words of length $n$ over $\{ 0, 1 \}$, and let
\[
X = \{ (x_1 \ x_2), \dots, (x_1 \ x_N), (x_1 \ x_N0 \ x_N1) \}.
\]
Note that $X \setminus \{ (x_1 \ x_N0 \ x_N1) \}$ generates the symmetric group on $A = \{ x_1, \dots, x_N \}$ and $X \setminus \{ (x_1 \ x_N) \}$ generates the symmetric group on $B = \{ x_1, \dots, x_{N-1}, x_N0, x_N1 \}$, so neither of these sets generate $V$, but Theorem~\ref{thm:criterion} implies that $X$ generates $V$.
We complete the proof by observing that for $2 \leq i \leq N-1$, every element of $X \setminus \{ (x_1 \ x_i) \}$ fixes $x_i\C$ pointwise, so $X \setminus \{ (x_1 \ x_i) \}$ does not generate $V$ either.
\end{proof}

\section{Torsion generating sets} 
\label{s:torsion}

\subsection{Main construction}
\label{ss:torsion_construction}

We now begin proving our main theorems.
In this section, we present a general construction, which we then we apply in the following section.

\begin{definition}
\label{def:subsets}
For $G \leq \Homeo(\C)$ and positive integers $m$ and $r$, we define sets
\[
\S_{m, r}(G) \subseteq \S_m(G) \subseteq \S(G) \subseteq G \times G
\]
as follows
\begin{align*}
\S(G)        &= \{ (\sigma, \tau) \in G \times G       \mid \text{$\sigma\tau$ has small support} \} \\
\S_m(G)      &= \{ (\sigma, \tau) \in \mathcal{S}(G)   \mid |\sigma\tau| = m \} \\
\S_{m, r}(G) &= \{ (\sigma, \tau) \in \mathcal{S}_m(G) \mid \text{$G = \< \sigma, \tau \>$ and $[\sigma^r, \tau^r] \neq 1$} \}.
\end{align*}
\end{definition}

Proposition~\ref{prop:torsion} guarantees that $\S_{m, r}(G)$ is nonempty whenever $G$ is a finitely generated simple vigorous group and $m, r \geq 2$.

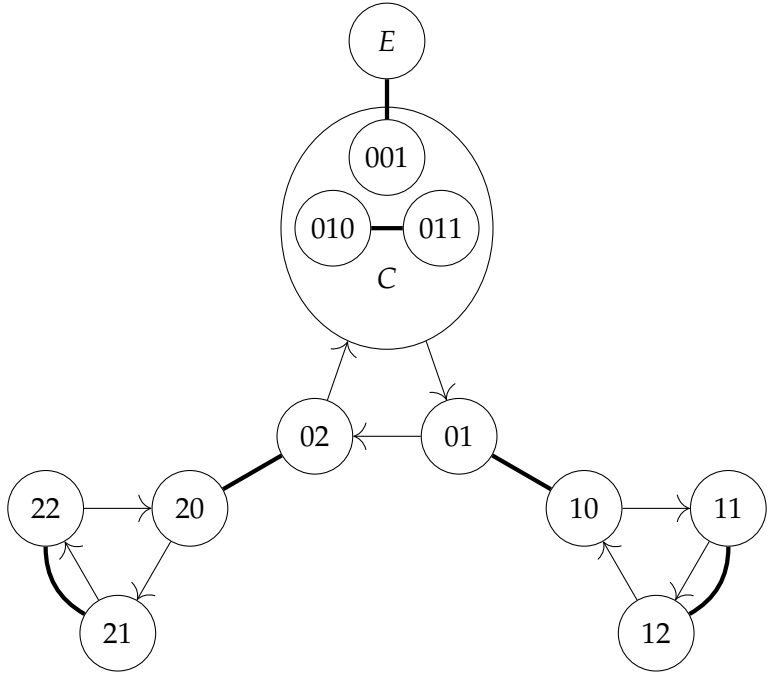
\begin{figure}[p]
\centering
\begin{tikzpicture}[scale=1.1, circle, minimum size=1cm]
  \foreach \x in {0,...,2}
  {
    \ifthenelse{\x = 0}
    {
      \node[draw,ellipse, minimum height=3.2cm, minimum width=2.8cm] (00)  at ($(90 : 1) + (0,0.5) + (0,0.5)$) {\raisebox{-1.6cm}{$C$}};
      \node[draw] (001) at ($(90 : 1.25) + (0,0.5) + (0,1.1)$) {001};
      \node[draw] (010) at ($(90 : 1.25) + (0,0.5) + (-0.65,0.25)$) {010};
      \node[draw] (011) at ($(90 : 1.25) + (0,0.5) + (0.65,0.25)$)  {011};
      \node[draw] (E)   at ($(90 : 2) + (0,0.5) + (0,1.75)$)  {$E$};
    }
    {
      \node[draw] (0\x) at (90 - \x * 120 : 1) {0\x};
      \foreach \y in {0,...,2}
      {
        \node[draw] (\x \y) at ($(90 - \x * 120 : {2 + 2 * sin(180/3)}) + (270 - \x * 120 - \y * 120 : 1)$) {\x \y};
      }
    }
  }
  \foreach \x in {0,...,2}
  {
    \pgfmathtruncatemacro{\xx}{mod(\x+1,3)}
    \draw[-{To[scale=2]}] (0\x) -- (0\xx);
    \ifthenelse{\x = 0}
    {
      \draw[ultra thick] (E)   edge (001);
      \draw[ultra thick] (010) edge (011); 
    }
    {
      \draw[ultra thick] (0\x) -- (\x0);
      \foreach \y in {0,...,2}
      {
        \ifthenelse{\y = 1}{
          \draw[ultra thick, bend left] (\x1) edge (\x2);
        }{}
        \pgfmathtruncatemacro{\yy}{mod(\y+1,3)}
        \draw[-{To[scale=2]}] (\x\y) edge (\x\yy);
      }
    }
  }
\end{tikzpicture}
\caption{The elements $\alpha$ (bold) and $\beta$ (thin) in Definition~\ref{def:technical} with $(m, n) = (2, 3)$.}
\label{fig:torsion}
\end{figure}

\begin{figure}[p]
\centering
\begin{tikzpicture}[scale=1.45, circle, minimum size=1cm]
  \foreach \x in {0,...,4}
  {
    \ifthenelse{\x = 0}
    {
      \node[draw,ellipse, minimum height=4.2cm, minimum width=3.5cm] (00)  at ($(90 : 1) + (0,1.5)$) {\raisebox{0.75cm}{$C$}};
      \node[draw] (E)   at ($(90 : 2)    + (0,0.25) + (0,2.25)$)     {$E$};
      \node[draw] (001) at ($(90 : 1.25) + (0,0.25) + (-0.6,1.65)$) {001};
      \node[draw] (002) at ($(90 : 1.25) + (0,0.25) + ( 0.6,1.65)$)  {002};
      \node[draw] (010) at ($(90 : 1.25) + (0,0.25) + (-0.6,0.9)$)  {010};
      \node[draw] (011) at ($(90 : 1.25) + (0,0.25) + ( 0.6,0.9)$)  {011};
      \node[draw] (012) at ($(90 : 1.25) + (0,0.25)$)                {012};
    }
    {
    \node[draw] (0\x)  at (90 - \x * 72 : 1)   {0\x};
    \ifthenelse{\x = 1 \OR \x = 2}    
    {
      \node[draw] (\x02) at ($(90 - \x * 72 : 1.5) + (- \x * 72 : 0.866)$) {\x02};
      \node[draw] (\x12) at ($(90 - \x * 72 : {2 + 2 * sin(180/5)}) + (270 - \x * 72 - 1.5 * 72 : 1.732)$) {\x12};
      \foreach \y in {0,...,4}
      {
        \node[draw] (\x \y) at ($(90 - \x * 72 : {2 + 2 * sin(180/5)}) + (270 - \x * 72 - \y * 72 : 1)$) {\x \y};
      }
    }
    {
    }
    }
  }
  \foreach \x in {0,...,4}
  {
    \pgfmathtruncatemacro{\xx}{mod(\x+1,5)}
    \draw[-{To[scale=2]}] (0\x) -- (0\xx);
    \ifthenelse{\x = 0}
    {
      \draw[-{To}, ultra thick] (E)   -- (001);
      \draw[-{To}, ultra thick] (001) -- (002);
      \draw[-{To}, ultra thick] (002) -- (E);      
      \draw[-{To}, ultra thick] (010) -- (011);
      \draw[-{To}, ultra thick] (011) -- (012);
      \draw[-{To}, ultra thick] (012) -- (010);   
    }
    {
    \ifthenelse{\x = 1 \OR \x = 2}
    {
      \draw[-{To},ultra thick] (0\x) -- (\x0); 
      \draw[-{To},ultra thick] (\x0) -- (\x02);
      \draw[-{To},ultra thick] (\x02) -- (0\x);
      \foreach \y in {0,...,4}
      {
        \ifthenelse{\y = 1}{
          \draw[-{To}, ultra thick, bend left] (\x1) edge (\x2);
          \draw[-{To}, ultra thick] (\x2) -- (\x12);
          \draw[-{To}, ultra thick] (\x12) -- (\x1);
        }{}
        \pgfmathtruncatemacro{\yy}{mod(\y+1,5)}
        \draw[-{To[scale=2]}] (\x\y) edge (\x\yy);
      }
    }
    {
    }
    }
  }
\end{tikzpicture}
\caption{The elements $\alpha$ (bold) and $\beta$ (thin) in Definition~\ref{def:technical} with $(m, n) = (3, 5)$.}
\label{fig:torsion_35}
\end{figure}

\begin{definition}
\label{def:technical}
Let $m \geq 2$ and $n \geq 3$ be integers.
Let $G \leq \Homeo(\C)$ be a simple vigorous group.
Let $C$ be a clopen set such that $\emptyset \subsetneq C \subsetneq \C$ and $[C]_G$ is even.
Let $(\sigma, \tau) \in \S(G_C)$.
We define $\Phi_{m, n}(G, C, \sigma, \tau)$ to be a pair $(\alpha, \beta) \in G \times G$ constructed as follows. 
(There are choices involved in the following construction and these can be made arbitrarily.)

For brevity, write $m' = m - 1$ and $n' = n - 1$.
Define 
\[
X = (\{ 10, 11, 20, 21 \} \times \{ 2, \dots, m' \}) \sqcup (\{ 0, 1, 2 \} \times \{ 0, \dots, n' \}).
\]
Let $\gamma_{00}$ be the identity of $G$, and for $x \in X \setminus \{00\}$, fix $\gamma_x \in G$ such that $\C = \bigsqcup_{x \in X} C\gamma_x \sqcup E$ for a nonempty clopen set $E$.
Define
\[
Y = (\{ 00, 01 \} \times \{ 0, \dots, m' \}) \setminus \{ 000 \}.
\] 
For $y \in Y$, fix $\gamma_y \in G$ such that $\bigsqcup_{y \in Y} E\gamma_y \subseteq C \setminus \Supp(\sigma\tau)$.

We now repeatedly apply Lemma~\ref{lem:existence}.
Let 
\[
\alpha = \tau^{-\gamma_{11}}\tau^{\gamma_{12}}\sigma^{-\gamma_{21}}\sigma^{\gamma_{22}}\pi_1\pi_2\pi_3
\]
where $\pi_1, \pi_2, \pi_3 \in G$, respectively, enact the permutations 
\begin{gather*}
(C\gamma_{01} \ C\gamma_{10} \ C\gamma_{102} \ \cdots \ C\gamma_{10m'})(C\gamma_{02} \ C\gamma_{20} \ C\gamma_{202} \ \cdots \ C\gamma_{20m'}) \\
(C\gamma_{12} \ C\gamma_{11} \ C\gamma_{112} \ \cdots \ C\gamma_{11m'})(C\gamma_{22} \ C\gamma_{21} \ C\gamma_{212} \ \cdots \ C\gamma_{21m'}) \\
(E \ E\gamma_{001} \ \cdots \ E\gamma_{00m'})(E\gamma_{010} \ E\gamma_{011} \ \cdots \ E\gamma_{01m'}).
\end{gather*}
Let $\beta \in G$ enact the permutation
\[
(C\gamma_{00} \ \cdots \ C\gamma_{0n'})(C\gamma_{10} \ \cdots \ C\gamma_{1n'})(C\gamma_{20} \ \cdots \ C\gamma_{2n'}).
\]
\end{definition}

Figures~\ref{fig:torsion} and~\ref{fig:torsion_35} illustrate Definition~\ref{def:technical} when $(m,n)$ is $(2, 3)$ and $(3, 5)$. In these figures, a region labelled by $u \in \mathbb{N}^*$ refers to the clopen set $C\gamma_u$ or $E\gamma_u$ as appropriate.

We now establish various properties of the elements constructed in Definition~\ref{def:technical}. 
Recall that a generating set $X$ of a group $G$ is \emph{minimal} if no proper subset of $X$ generates $G$.
For the arguments in the following section a stronger condition will be convenient.
A generating set $X$ of a vigorous group $G \leq \Homeo(\C)$ is \emph{strongly minimal} if every proper subset of $X$ fixes a nonempty clopen subset of $\C$ pointwise.
Since a vigorous group $G \leq \Homeo(\C)$ fixes no nonempty clopen subset of $\C$ pointwise, a strongly minimal generating set of $G$ is  minimal.

It will be convenient to define $r \: \Z \times \Z \to \Z$ as 
\[
(m, n) \mapsto  m(4m + 3n - 10)((m - 1)(4m + 3n - 10)+1).
\]

\begin{proposition}
\label{prop:technical}
Let $m \geq 2$ and $n \geq 3$ be integers and write $r = r(m, n)$.
Let $G \leq \Homeo(\C)$ be a simple vigorous group.
Let $C$ be a clopen set such that $\emptyset \subsetneq C \subsetneq \C$ and $[C]_G$ is even.
Let $(\sigma, \tau) \in \S(G_C)$ and let $(\alpha, \beta) = \Phi_{m, n}(G, C, \sigma, \tau)$.
Then the following hold
\begin{enumerate}
\item $\alpha$ has small support and order $m$
\item $\beta$ has small support and order $n$
\item $(\alpha\beta)^{\beta^{-2}\alpha^{-1}\beta^{-2}}$ setwise stabilises $C$ and $(\alpha\beta)^{\beta^{-2}\alpha^{-1}\beta^{-2}}|_C = \sigma|_C$
\item $(\alpha\beta)^{\beta^{-2}\alpha^{-1}\beta^{-1}}$ setwise stabilises $C$ and $(\alpha\beta)^{\beta^{-2}\alpha^{-1}\beta^{-1}}|_C = \tau|_C$ 
\item if $(\sigma, \tau) \in \S_m(G_C)$, then $[((\alpha\beta)^{\beta^{-2}\alpha^{-1}\beta^{-2}})^r, ((\alpha\beta)^{\beta^{-2}\alpha^{-1}\beta^{-1}})^r] = [\sigma^r,\tau^r]$
\item if $(\sigma, \tau) \in \S_{m, r}(G_C)$, then $\{\alpha, \beta\}$ is a strongly minimal generating set of $G$ of size two
\item if $(\sigma, \tau) \in \S_{m, r}(G_C)$, then there are involutions $\gamma, \delta \in G$ satisfying $\beta = \gamma\delta$ such that $\{\alpha, \gamma, \delta \}$ is a strongly minimal generating set of $G$ of size three.
\end{enumerate}
\end{proposition}

\begin{proof}
We adopt the notation of Definition~\ref{def:technical}.
Inspecting the construction in Definition~\ref{def:technical}, it is easy to see that~(i) and~(ii) hold.
We also see that $(\alpha\beta)|_{C\gamma_{12}} = \tau^{\gamma_{12}}|_{C\gamma_{12}}$ and $(\alpha\beta)|_{C\gamma_{22}} = \sigma^{\gamma_{22}}|_{C\gamma_{22}}$, and, by consulting Figure~\ref{fig:torsion}, that
\begin{equation}
\gamma_{12} \beta^{-2}\alpha^{-1}\beta^{-1}|_C = \gamma_{22} \beta^{-2}\alpha^{-1}\beta^{-2}|_C = \gamma_{00}|_C = 1_C.
\label{eq:move}
\end{equation}
This establishes~(iii) and~(iv).
In particular, if $H$ is the setwise stabiliser of $C$ in $\<\alpha, \beta \>$, then
\begin{equation}
H|_C \geq (G_C)|_C.
\label{eq:induce}
\end{equation}

For the rest of the proof, assume that $(\sigma, \tau) \in \S_m(G_C)$. 
We claim that 
\begin{equation}
(\alpha\beta)^r = (\tau^{\gamma_{12}}\sigma^{\gamma_{22}})^r.
\label{eq:power}
\end{equation}
To prove this, write $s = 4m + 3n - 10$, noting that $r = ms((m-1)s+1)$, and write
\[
D = C \setminus \bigsqcup_{y \in Y} E\gamma_y.
\]
Then, restricted to $D$, we see $(\alpha\beta)^s = (\tau^{-1}\sigma^{-1})$ (again, the reader might find it helpful to consult Figure~\ref{fig:torsion}, recalling that $\gamma_{00} = 1$). 
Since $|\tau^{-1}\sigma^{-1}| = |\sigma\tau| = m$, we deduce that, restricted to $D$, the element $(\alpha\beta)^{sm}$ is trivial.
Since
\[
\left(\bigsqcup_{i=1}^{m'} E\gamma_{00i} \sqcup \bigsqcup_{i=0}^{m'} E\gamma_{01i}\right) \subseteq \Fix_{C}(\sigma\tau),
\] 
restricted to $\bigsqcup_{i=0}^{m'} E\gamma_{01i}$, the element $(\alpha\beta)^s$ enacts $(E\gamma_{010} \, \dots \, E\gamma_{01m'})$, so $(\alpha\beta)^{sm}$ is trivial, and restricted to $\bigsqcup_{i=1}^{m'} E\gamma_{00i}$, the element $(\alpha\beta)^{(m-1)s+1}$ is trivial. 
Therefore, $(\alpha\beta)^r \in G_{(C)}$. 
Now, it is easy to check that $\C \setminus \left(C\gamma_{12} \sqcup C\gamma_{22}\right) \subseteq C\<\alpha\beta\>$.
Therefore, $(\alpha\beta)^r \in G_{C\gamma_{12} \sqcup C\gamma_{22}}$, and we conclude that $(\alpha\beta)^r = (\tau^{\gamma_{12}} \sigma^{\gamma_{22}})^r$, as claimed.

Applying~\eqref{eq:move} and~\eqref{eq:power}, we see that
\begin{align}
[((\alpha\beta)^{\beta^{-2}\alpha^{-1}\beta^{-2}})^r, ((\alpha\beta)^{\beta^{-2}\alpha^{-1}\beta^{-1}})^r]
     &= [((\tau^{\gamma_{12}} \sigma^{\gamma_{22}})^{\beta^{-2}\alpha^{-1}\beta^{-2}})^r, ((\tau^{\gamma_{12}} \sigma^{\gamma_{22}})^{\beta^{-2}\alpha^{-1}\beta^{-1}})^r] \nonumber \\
     &= [( \tau^{\gamma_{0n'}})^r(\sigma^{\gamma_{00}})^r, (\tau^{\gamma_{00}})^r(\sigma^{\gamma_{01}})^r] \nonumber \\
     &= [\sigma^r,\tau^r]^{\gamma_{00}} \nonumber \\
     &= [\sigma^r,\tau^r], \label{eq:commutator}
\end{align}
which proves~(v).

For the rest of the proof, assume further that $(\sigma, \tau) \in \S_{m, r}(G_C)$.
We claim that $G_C \leq \< \alpha, \beta \>$.
Since $G_C = \< \sigma, \tau \>$, by~\eqref{eq:induce},
\[
S = \{ ([\sigma^r,\tau^r])^\theta \mid \theta \in \< \alpha, \beta \> \} = \{ ([\sigma^r,\tau^r])^\theta \mid \theta \in G_C \}.
\] 
Since  $[\sigma^r,\tau^r] \neq 1$, we know that $S$ is a nontrivial normal subset of $G_C$, but $G_C$ is simple, so we deduce that $\< S \> = G_C$. 
By~\eqref{eq:commutator}, $[\sigma^r,\tau^r] \in \< \alpha, \beta \>$, so $S \subseteq \< \alpha, \beta \>$, which means that
\begin{equation}
G_C = \< S \> \leq \< \alpha,\beta \>.
\label{eq:local_1}
\end{equation}

We now claim that $G_{C \sqcup E} \leq \< \alpha, \beta \>$ also.
Since $E\gamma_{001} \sqcup E\gamma_{011} \subseteq C$, we have 
\[
G_{E\gamma_{001} \sqcup E\gamma_{011}} \leq G_C \leq \< \alpha, \beta \>.
\] 
In addition, since $(E\gamma_{001} \sqcup E\gamma_{011})\beta^{-1} = E \sqcup E_{\gamma_{010}}$, we have
\[
G_{E \sqcup E\gamma_{010}} = G_{E\gamma_{001} \sqcup E\gamma{011}}^{\beta^{-1}} \leq \< \alpha, \beta \>.
\] 
Since $C \cap (E \sqcup E_{\gamma_{010}}) \neq \emptyset$, by Lemma~\ref{lem:union}, 
\begin{equation}
G_{C \sqcup E} = \< G_C, G_{E \sqcup E_{\gamma_{010}}} \> = \< G_C, G_{E\gamma_{001} \sqcup E_{\gamma_{011}}}^{\beta^{-1}} \> \leq \< \alpha, \beta \>.
\label{eq:local_2}
\end{equation} 

We are now in a position to prove that $G = \< \alpha, \beta \>$.
Define 
\[
\Gamma = \{ \beta, \beta^2, \beta\alpha\beta, \beta^2\alpha\beta \}A \cup \{ 1, \beta\alpha, \beta^2\alpha \}B
\]
where $A = \{ \alpha^i \mid 0 \leq i \leq m' \}$ and $B = \{ \beta^i \mid 0 \leq i \leq n' \}$.
Then it is easy to check that
\[
\C \setminus E = \bigsqcup_{\gamma \in \Gamma} C\gamma.
\]
In particular,
\begin{equation}
\C = \bigcup_{\gamma \in \Gamma} C\gamma \cup \bigcup_{\gamma \in \Gamma} (C \sqcup E)\gamma.
\label{eq:cover}
\end{equation}
Let $\Delta$ be the graph with vertex set 
\[
\{ C\gamma \mid \gamma \in \Gamma \} \sqcup \{ (C \sqcup E)\gamma \mid \gamma \in \Gamma \} 
\] 
and where $u$ and $v$ are adjacent if and only if $u \cap v \neq \emptyset$. 
We claim that $\Delta$ is connected.
First, $C \subseteq C \sqcup E$, so $C\gamma \subseteq (C \sqcup E)\gamma$ for all $\gamma \in G$.
In particular, $C\gamma$ is adjacent to $(C \sqcup E)\gamma$ for all $\gamma \in \Gamma$.
Next, for all integers $i$ and $j$, note that $E\beta^i\alpha^j \subseteq C \sqcup E$, so $E\beta^i\alpha^j\gamma \subseteq (C \sqcup E)\gamma$ for all $\gamma \in G$.
In particular, $(C \sqcup E)\beta\alpha\beta^i$ and $(C \sqcup E)\beta^2\alpha\beta^i$ are adjacent to $(C \sqcup E)\beta^i$ for all $i$, and $(C \sqcup E)\beta\alpha\beta\alpha^i$ and $(C \sqcup E)\beta^2\alpha\beta\alpha^i$ are adjacent to $(C \sqcup E)\beta\alpha^i$ for all $i$.
In addition, $(C \sqcup E)\beta^i$, $(C \sqcup E)\beta\alpha^i$ and $(C \sqcup E)\beta^2\alpha^i$ are adjacent to $C \sqcup E$ for all $i$.
This proves that $\Delta$ is connected. 
Now, by repeatedly applying Lemma~\ref{lem:union} and using \eqref{eq:cover}, we conclude that 
\[
G = \< G_C, G_{C \sqcup E}, \Gamma \>.
\]
By \eqref{eq:local_1} and \eqref{eq:local_2}, $G_C, G_{C \sqcup E} \leq \< \alpha, \beta\>$ and, by construction, $\Gamma \subseteq \< \alpha, \beta \>$, so we conclude that $G = \< \alpha, \beta \>$, as claimed.
Since $\alpha$ and $\beta$ have small support, $\{ \alpha, \beta \}$ is a strongly minimal generating set of size two, which establishes~(vi).

It remains to prove~(vii).
Every element of a finite symmetric group is a product of two involutions, and since $\beta$ enacts the permutation 
\[
(C\gamma_{00} \ \cdots \ C\gamma_{0n'})(C\gamma_{10} \ \cdots \ C\gamma_{1n'})(C\gamma_{20} \ \cdots \ C\gamma_{2n'})
\]
and $[C]_G$ is even, we quickly deduce that $\beta$ is the product of two involutions in $G$.
However, we will specify a particular decomposition of $\beta$ into two involutions in order to guarantee our desired strong minimality condition. 

Note that $\Fix_C(\beta) = D$. 
Let $C_1$ be a clopen set with $[C_1]_G$ even such that $\emptyset \subsetneq C_1 \subsetneq C$ and $\emptyset \subsetneq C_1 \cap D \subsetneq D$, and write $C_2 = C \setminus C_1$.
In particular, $C\gamma_{00} = C = C_1 \sqcup C_2$, so $C_1 \cap D$ and $C_2 \cap D$ are both nonempty.
Applying Lemma~\ref{lem:existence}, let $\gamma \in G$ enact the permutation
\begin{gather*}
(C_1\gamma_{00} \ C_1\gamma_{0n'})(C_1\gamma_{01} \ C_1\gamma_{0(n'-1)}) \cdots (C_1\gamma_{0(\lfloor n/2 \rfloor - 1)} \ C_1\gamma_{0\lceil n/2 \rceil}) \\
(C_2\gamma_{0n'} \ C_2\gamma_{01})(C_2\gamma_{0(n'-1)} \ C_2\gamma_{02}) \cdots (C_2\gamma_{0(\lfloor n/2 \rfloor + 1)} \ C_2\gamma_{0(\lceil n/2 \rceil - 1)}) \\
(C\gamma_{10} \ C\gamma_{1n'})(C\gamma_{11} \ C\gamma_{1(n'-1)}) \cdots (C\gamma_{1(\lfloor n/2 \rfloor - 1)} \ C\gamma_{1\lceil n/2 \rceil}) \\
(C\gamma_{20} \ C\gamma_{2n'})(C\gamma_{21} \ C\gamma_{2(n'-1)}) \cdots (C\gamma_{2(\lfloor n/2 \rfloor - 1)} \ C\gamma_{2\lceil n/2 \rceil}),
\end{gather*}
and let $\delta \in G$ enact the permutation
\begin{gather*}
(C_1\gamma_{0n'} \ C_1\gamma_{01})(C_1\gamma_{0(n'-1)} \ C_1\gamma_{02}) \cdots (C_1\gamma_{0(\lfloor n/2 \rfloor + 1)} \ C_1\gamma_{0(\lceil n/2 \rceil - 1)}) \\
(C_2\gamma_{00} \ C_2\gamma_{0n'})(C_2\gamma_{01} \ C_2\gamma_{0(n'-1)}) \cdots (C_2\gamma_{0(\lfloor n/2 \rfloor - 1)} \ C_2\gamma_{0\lceil n/2 \rceil}) \\
(C\gamma_{1n'} \ C\gamma_{11})(C\gamma_{1(n'-1)} \ C\gamma_{12}) \cdots (C\gamma_{1(\lfloor n/2 \rfloor + 1)} \ C\gamma_{1(\lceil n/2 \rceil - 1)}) \\
(C\gamma_{2n'} \ C\gamma_{21})(C\gamma_{2(n'-1)} \ C\gamma_{22}) \cdots (C\gamma_{2(\lfloor n/2 \rfloor + 1)} \ C\gamma_{2(\lceil n/2 \rceil - 1)}).
\end{gather*}
Then it is easy to see that $\beta = \gamma\delta$.
In particular, $G = \< \alpha, \beta \> = \< \alpha, \gamma, \delta \>$.
Moreover, $\gamma$ and $\delta$ both fix $E$, and, by construction, $\alpha$ and $\gamma$ both fix $C_2 \cap D$, and $\alpha$ and $\delta$ both fix $C_1 \cap D$.
Therefore, $\{ \alpha, \gamma, \delta \}$ is a strongly minimal generating set of size three, which establishes (vii) and completes the proof.
\end{proof}

\begin{proposition}
\label{prop:injective}
Let $m \geq 2$ and $n \geq 3$ be integers and write $r = r(m, n)$.
Let $G \leq \Homeo(\C)$ be a simple vigorous group.
Let $C$ be a clopen set such that $\emptyset \subsetneq C \subsetneq \C$ and $[C]_G$ is even.
For $i \in \{1, 2\}$, let $(\sigma_i, \tau_i) \in \S_{m, r}(G_C)$ and let $(\alpha_i, \beta_i) = \Phi_{m, n}(G, C, \sigma_i, \tau_i)$.
Then the following hold
\begin{enumerate}
\item if $(\alpha_1, \beta_1) = (\alpha_2, \beta_2)$, then $(\sigma_1, \tau_1) = (\sigma_2, \tau_2)$
\item if $\phi \in \Aut(G)$ satisfies $(\alpha_1, \beta_1) = (\alpha_2\phi, \beta_2\phi)$, then $G_C \phi = G_C$ and $(\sigma_1, \tau_1) = (\sigma_2\phi, \tau_2\phi)$.
\end{enumerate}
\end{proposition}

\begin{proof}
Define the words $u, v, w \in \{ a, a^{-1}, b, b^{-1} \}^*$ as
\[
u = (ab)^{b^{-2}a^{-1}b^{-2}} \qquad
v = (ab)^{b^{-2}a^{-1}b^{-1}} \qquad
w = [((ab)^{b^{-2}a^{-1}b^{-2}})^r, ((ab)^{b^{-2}a^{-1}b^{-1}})^r].
\]
For any word $x \in \{ a, a^{-1}, b, b^{-1} \}^*$, there is a natural map $G \times G \to G$, which we also denote by $x$, where $(\gamma, \delta)x$ is defined by substituting $\gamma$ and $\delta$ for $a$ and $b$ respectively.

First assume that $(\alpha_1, \beta_1) = (\alpha_2, \beta_2)$.
Then, by Proposition~\ref{prop:technical}(iii) and~(iv), 
\begin{gather*}
\sigma_1|_C = (\alpha_1, \beta_1)u|_C = (\alpha_2, \beta_2)u|_C = \sigma_2|_C \quad \text{and} \quad
\tau_1|_C = (\alpha_1, \beta_1)v|_C = (\alpha_2, \beta_2)v|_C = \tau_2|_C,
\end{gather*}
which proves~(i) since $\sigma_1, \sigma_2, \tau_1, \tau_2 \in G_C$.

Now assume that $\phi \in \Aut(G)$ satisfies $(\alpha_1, \beta_1) = (\alpha_2\phi, \beta_2\phi)$.
For $i \in \{1, 2\}$, write
\[
u_i = (\alpha_i, \beta_i)u \qquad v_i = (\alpha_i, \beta_i)v \qquad w_i = (\alpha_i, \beta_i)w.
\]
Let $i \in \{1,2\}$.
Then $\< u_i, v_i \>$ setwise stabilises $C$ and $\< u_i, v_i \>|_C = (G_C)|_C$. 
Moreover, by Proposition~\ref{prop:technical}(v), $w_i = [\sigma_i^r, \tau_i^r]$, which is a nontrivial element of the simple group $G_C$, so $\< w_i^{\< u_i, v_i \>} \> = G_C$. 
However, $(\alpha_1, \beta_1) = (\alpha_2\phi, \beta_2\phi)$, so $u_1 = u_2\phi$, $v_1 = v_2\phi$ and $w_1 = w_2\phi$.
Therefore,
\[
G_C = \< w_1^{\< u_1, v_1 \>} \> = \< w_2^{\< u_2, v_2 \>} \>\phi = G_C\phi,
\]
as claimed.
Let $\psi = \phi|_{(G_C)|_C}$ and note that $\psi \in \Aut((G_C)|_C)$.
Then
\[
\sigma_1|_C = u_1|_C = (u_2\phi)|_C = (u_2|_C)\psi = (\sigma_2|_C)\psi = (\sigma_2\phi)|_C,
\]
which proves $\sigma_1 = \sigma_2\phi$ since $\sigma_1, \sigma_2\phi \in G_C$.
A similar argument gives $\tau_1 = \tau_2\phi$, which completes the proof of~(ii).
\end{proof}

\subsection{Applications}
\label{ss:torsion_applications}

We now use the construction in Section~\ref{ss:torsion_construction} to prove a number of our main theorems.

With a view towards Theorem~\ref{thm:products}, it will be useful to work with generating sequences rather than generating sets.
To be precise, let $G$ be a group.
We say that $(\alpha_1, \dots, \alpha_k) \in G^k$ is a \emph{generating sequence} of $G$ if $\{\alpha_1, \dots, \alpha_k\}$ is a generating set of $G$, and a generating sequence with distinct entries is \emph{(strongly) minimal} if the corresponding generating set is (strongly) minimal.
Additionally, sequences $(\alpha_1, \dots, \alpha_k), (\beta_1, \dots, \beta_k) \in G^k$ are \emph{$\Aut(G)$-equivalent} if there exists $\phi \in \Aut(G)$ such that $(\alpha_1, \dots, \alpha_k) = (\beta_1\phi, \dots, \beta_k\phi)$ and \emph{$\Aut(G)$-inequivalent} otherwise.

Our first result in this section immediately implies Theorems~\ref{thm:torsion} and~\ref{thm:involutions}.

\begin{theorem}
\label{thm:torsion_involutions_strong}
Let $G$ be a finitely generated simple vigorous group.
\begin{enumerate}
\item Let $m \geq 2$ and $n \geq 3$. Then there are infinitely many pairwise $\Aut(G)$-inequivalent strongly minimal generating sequences $(\alpha, \beta)$ of $G$ such that  $|\alpha| = m$ and $|\beta| = n$.
\item There are infinitely many pairwise $\Aut(G)$-inequivalent strongly minimal generating sequences of $G$ consisting of three involutions.
\end{enumerate}
\end{theorem}

\begin{proof}
Let $m \geq 2$ and $n \geq 3$, and write $r = r(m, n)$.
Let $C$ be a clopen set such that $\emptyset \subsetneq C \subsetneq \C$ and $[C]_G$ is even.
By Lemma~\ref{lem:restriction}, $G_C$ is also a finitely generated simple vigorous group, so, by Proposition~\ref{prop:torsion}, there are integers $0 < l_1 < l_2 < \dots$ such that for each $i$ there exists $(\sigma_i, \tau_i) \in \S_{m, r}(G_C)$ with $|\sigma_i| = l_i$.
For each positive integer $i$, write $(\alpha_i, \beta_i) = \Phi_{m, n}(G, C, \sigma_i, \tau_i)$.

By Proposition~\ref{prop:technical}(vi), $(\alpha_i, \beta_i)$ is a strongly minimal generating sequence for each $i$.
To prove (i), it suffices to prove that $(\alpha_1, \beta_1)$, $(\alpha_2, \beta_2)$, \dots are pairwise $\Aut(G)$-inequivalent.
To see this, fix $i < j$ and suppose that $(\alpha_i, \beta_i) = (\alpha_j\phi, \beta_j\phi)$ for $\phi \in \Aut(G)$.
Then, by Proposition~\ref{prop:injective}(ii), $(\sigma_i, \tau_i) = (\sigma_j\phi, \tau_j\phi)$, which is impossible, since $|\sigma_i| = l_i < l_j = |\sigma_j| = |\sigma_j\phi|$. 

By Proposition~\ref{prop:technical}(vii), for each $i$, we can write $\beta_i = \gamma_i\delta_i$ for involutions $\gamma_i, \delta_i \in G$ such that $(\alpha_i, \gamma_i, \delta_i)$ is a strongly minimal generating sequence, which, specialising to the case $m = 2$, proves (ii), noting that $(\alpha_1, \gamma_1, \delta_1)$,  $(\alpha_2, \gamma_2, \delta_2)$, \dots are pairwise $\Aut(G)$-inequivalent since $(\alpha_1, \gamma_1\delta_1) = (\alpha_1, \beta_1)$, $(\alpha_2, \gamma_2\delta_2) = (\alpha_2, \beta_2)$, \dots are pairwise $\Aut(G)$-inequivalent.
\end{proof}

Theorem~\ref{thm:torsion_involutions_strong} forms a base for an inductive proof of the following more general theorem, which immediately implies Theorem~\ref{thm:minimal}.

\begin{theorem}
\label{thm:minimal_strong}
Let $G \leq \Homeo(\C)$ be a finitely generated simple vigorous group. 
Let $k \geq 2$ and let $m_1 \geq \cdots \geq m_k \geq 2$ be integers where $\{ m_1, m_2 \} \neq \{ 2 \}$ if $k = 2$.
Then there are infinitely many pairwise $\Aut(G)$-inequivalent strongly minimal generating sequences $(\alpha_1, \dots, \alpha_k)$ of $G$ such that $|\alpha_i| = m_i$ for all $1 \leq i \leq k$.
\end{theorem}

\begin{proof}
We proceed by induction on $k$.
For the base case where $k = 2$, by Theorem~\ref{thm:torsion_involutions_strong}(i), $G$ has infinitely many pairwise $\Aut(G)$-inequivalent strongly minimal generating sequences $(\alpha_1, \alpha_2)$ where $|\alpha_1| = m_1$ and $|\alpha_2| = m_2$.
Now assume that $k > 2$ and the result holds for smaller $k$.
If $k = 3$ and $m_1 = m_2 = m_3 = 2$, then, by Theorem~\ref{thm:torsion_involutions_strong}(ii), $G$ has infinitely many pairwise $\Aut(G)$-inequivalent strongly minimal generating sequences $(\alpha_1, \alpha_2, \alpha_3)$ where $|\alpha_1| = |\alpha_2| = |\alpha_3| = 2$.
Hence, we can assume that $k \geq 3$ and $\{ m_1, m_2, m_3 \} \neq \{ 2 \}$ if $k = 3$.

Let $C$ be a clopen set such that $\emptyset \subsetneq C \subsetneq \C$ and $[C]_G$ is even.
Let $\gamma_1, \dots, \gamma_{{m_k}-1} \in G$ such that $C\gamma_i \cap C\gamma_j \neq \emptyset$ if and only if $i = j$ and $\bigsqcup_{i=1}^{{m_k}-1} C\gamma_i \subsetneq \C$.
Let $D = \C \setminus \bigsqcup_{i=1}^{{m_k}-1} C\gamma_i$.
By Lemma~\ref{lem:restriction}, $G_D$ is a finitely generated simple vigorous group.
Since $k - 1 \geq 2$ and $\{m_1, m_2 \} \neq \{ 2 \}$ if $k - 1 = 2$, by induction, $G_D$ has infinitely many pairwise $\Aut(G)$-inequivalent strongly minimal generating sequences $(\alpha_1, \dots, \alpha_{k-1})$ such that $|\alpha_i| = m_i$ for each $1 \leq i \leq k-1$.

Fix one such sequence $X_0 = (\alpha_1, \dots, \alpha_{k-1})$.
Then, for each $1 \leq i \leq k-1$, we can fix a nonempty clopen set $A_i \subseteq D$ that is fixed by $\{\alpha_1, \dots, \alpha_{k-1}\} \setminus \{ \alpha_i \}$. 
Let $E \subsetneq D$ such that $A_i \not\subseteq E$ for all $1 \leq i \leq k-1$ and let $\gamma_{m_k} \in G$ such that $C\gamma_{m_k} \subseteq E$.
By Lemma~\ref{lem:existence}, let $\alpha_k$ enact the permutation $(C\gamma_1 \ \cdots \ C\gamma_{m_k})$ and note that $|\alpha_k| = m_k$.
Write $X = (\alpha_1, \dots, \alpha_k)$.

We claim that $X$ is a generating sequence of $G$.
Recall that $\C = D \sqcup \bigsqcup_{i=1}^{m_{k-1}} C\gamma_i$ and $C\gamma_{m_k} \subseteq E \subsetneq D$.
Therefore, since $\alpha_k$ enacts $(C\gamma_1 \ \cdots \ C\gamma_{m_k})$, we deduce that $\C = \bigcup_{i=0}^{m_k-1} D\alpha_k^i$ and $D \cap D\alpha_k^i \neq \emptyset$ for all $0 \leq i \leq m_k-1$. 
Hence, by repeated application of Lemma~\ref{lem:union}, we see that $G = \< G_D, \alpha_k \>$. 
However, $G_D = \< \alpha_1, \dots, \alpha_{k-1} \>$, so $G = \< \alpha_1, \dots, \alpha_k \>$, as claimed.

We now claim that $X$ is strongly minimal.
Certainly, $\{ \alpha_1, \dots, \alpha_{k-1} \}$ fixes the clopen set $\bigsqcup_{i=1}^{{m_k}-1} C\gamma_i \subsetneq \C$ pointwise.
Now let $1 \leq i \leq k-1$ and consider $\{ \alpha_1, \dots, \alpha_k \} \setminus \{ \alpha_i \}$.
By construction, $A_i \setminus (E \cap A_i)$ is nonempty and fixed pointwise by $\{ \alpha_1, \dots, \alpha_{k-1} \} \setminus \{ \alpha_i \}$, and it is also fixed pointwise by $\alpha_k$ since $\Supp(\alpha_k) \cap D = C\gamma_{m_k} \subseteq E$.
This proves the claim.

Finally, let $X_0 = (\alpha_1, \dots, \alpha_{k-1})$ and $Y_0 = (\beta_1, \dots, \beta_{k-1})$ be $\Aut(G_D)$-inequivalent strongly minimal generating sequences of $G_D$ and let $X = (\alpha_1, \dots, \alpha_k)$ and $Y = (\beta_1, \dots, \beta_k)$ be strongly minimal generating sequences for $G$ obtained by the previous construction.
We claim that $X$ and $Y$ are $\Aut(G)$-inequivalent.
To prove this, suppose otherwise and fix $\phi \in \Aut(G)$ such that $(\alpha_1, \dots, \alpha_k) = (\beta_1\phi, \dots, \beta_k\phi)$. 
Then $\alpha_i = \beta_i\phi$ for all $1 \leq i \leq k-1$ and $G_D = \< \alpha_1, \dots, \alpha_{k-1} \> = \< \beta_1, \dots, \beta_{k-1} \>$, so $\phi$ restricted to $G_D$ is an element of $\Aut(G_D)$ that maps $Y_0 = (\beta_1, \dots, \beta_{k-1})$ to $(\alpha_1, \dots, \alpha_{k-1}) = X_0$, which contradicts $X_0$ and $Y_0$ being $\Aut(G_D)$-inequivalent.
This proves the claim and completes the induction.
\end{proof}

A benefit of the strong formulation of Theorem~\ref{thm:minimal_strong} is that it quickly implies Theorem~\ref{thm:products}.

\begin{proof}[Proof of Theorem~\ref{thm:products}]
Let $G$ be a finitely generated simple vigorous group and let $s$ be a positive integer.
We claim that $G^s$ has a minimal generating sequence $(\gamma_1, \dots, \gamma_k)$ where $|\gamma_j| = m_j$ for all $1 \leq j \leq k$.
By Theorem~\ref{thm:minimal_strong}, there exist pairwise $\Aut(G)$-inequivalent minimal generating sequences $X_1 = (\alpha_{11}, \dots, \alpha_{1k})$, \dots, $X_s = (\alpha_{s1}, \dots, \alpha_{sk})$ where $|\alpha_{ij}| = m_j$ for all $1 \leq i \leq s$ and $1 \leq j \leq k$. 
For $1 \leq j \leq k$, write $\gamma_j = (\alpha_{1j}, \dots, \alpha_{sj})$ and note that $|\gamma_j| = m_j$ since $|\alpha_{ij}| = m_j$ for all $1 \leq i \leq s$.
As $X_1$, \dots, $X_s$ are generating sequences of $G$, the subgroup $\< \gamma_1, \dots, \gamma_k \>$ is a subdirect product of $G^s$, and as $G$ is simple and $X_1$, \dots, $X_s$ are pairwise $\Aut(G)$-inequivalent, we conclude that $\< \gamma_1, \dots, \gamma_k \> = G^s$.
Finally, $(\gamma_1, \dots, \gamma_k)$ is a minimal generating sequence of $G^s$ since $X_1$ is a minimal generating sequence of $G$.
This proves the claim.

Now let $G = G_1 \times \cdots \times G_l$ as in the statement, and write $G = G_1^{s_1} \times \cdots \times G_t^{s_t}$ where $G_1, \dots, G_t$ are pairwise nonisomorphic and $s_1, \dots, s_t$ are positive integers. 
By the previous claim, for each $1 \leq i \leq t$, we can fix a minimal generating sequence $Y_i = (\gamma_{i1}, \dots, \gamma_{ik})$ such that $|\gamma_{ij}| = m_j$ for all $1 \leq j \leq k$. 
For $1 \leq j \leq k$, write $\delta_j = (\gamma_{1j}, \dots, \gamma_{tj})$ and note that $|\delta_j| = m_j$ since $|\gamma_{ij}| = m_j$ for all $1 \leq i \leq t$.
As $Y_1$, \dots, $Y_t$ are generating sequences of $G_1^{s_1}$, \dots, $G_t^{s_t}$, the subgroup $\< \delta_1, \dots, \delta_k \>$ is a subdirect product of $G_1^{s_1} \times \cdots \times G_t^{s_t}$, and as $G_1$, \dots, $G_t$ are pairwise nonisomorphic simple groups, for $i \neq j$ no composition factor of $G_i^{s_i}$ is isomorphic to $G_j$, so $\< \delta_1, \dots, \delta_k \> = G$.
Finally, $(\delta_1, \dots, \delta_k)$ is a minimal generating sequence of $G$ since $Y_1$ is a minimal generating sequence of $G_1^{s_1}$.
This completes the proof.
\end{proof}

We conclude this section by giving a simpler construction for Thompson's group $V$, proving the claim in Example~\ref{ex:torsion}.

\begin{corollary}
\label{cor:torsion_v}
We have $|\alpha| = 2$, $|\beta| = 3$ and $V = \< \alpha, \beta \>$ where
\begin{gather*}
\alpha = (00 \ 010100)(010110 \ 010111) (100 \ 0110)(110 \ 0111) \cdot (\phi\psi)_{[101]}^{-1} (\phi\psi)_{[111]} (101 \ 111) \\
\beta = (01 \ 10 \ 11)
\end{gather*}
for the elements
\begin{gather*}
\phi = (0000 \ 010 \ 00011 \ 00010)(0010 \ 011 \ 0011) \cdot (000 \ 001) \\
\psi = (1000 \ 10010 \ 10011 \ 110)(1010 \ 1011 \ 111).
\end{gather*}
\end{corollary}

Before proving Corollary~\ref{cor:torsion_v}, we first establish a preliminary result.

\begin{lemma}
\label{lem:torsion_v}
Let $\lambda = (00 \ 01)$ and $\mu = (000 \ 0010 \ 0011 \ 10)(010 \ 011 \ 11)$. 
Then $\< \lambda, \mu \> = V$.
\end{lemma}

\begin{proof}
By Theorem~\ref{thm:criterion}, it suffices to show that $\Alt(A) \leq \< \lambda, \mu \>$ and $\Alt(B) \leq \< \lambda, \mu \>$ for
\[
A = \{ 000, 001, 010, 011, 10, 11 \} \text{ and } B = \{ 000, 0010, 0011, 010, 011, 10, 11 \}.
\]

First note that $\lambda = (000 \ 010)(001 \ 011)$ and $\mu^4 = (010 \ 011 \ 11)$, so 
\[
\Alt(A \setminus \{ 10 \}) = \< (000 \ 010)(001 \ 011), (010 \ 011 \ 11) \> \leq \< \lambda, \mu \>.
\]
In particular, $\< \lambda, \mu \>$ contains every $3$-cycle in $\Alt(A \setminus \{ 10 \})$.
However, 
\[
(010 \ 011 \ 10) = (010 \ 011 \ 000)^{\mu^3} \in \< \lambda, \mu \>,
\] 
and conjugating by suitable elements of $\Alt(A \setminus \{ 10 \})$ shows that $\< \lambda, \mu \>$ contains every $3$-cycle in $\Alt(A)$, so $\Alt(A) \leq \< \lambda, \mu \>$.

Next note that $\mu^3 = (000 \ 10 \ 0011 \ 0010)$ and $(000 \ 10 \ 010 \ 011 \ 11) \in \Alt(A)$, so
\[
\Alt(B) \leq \Sym(B) = \< (000 \ 10 \ 0011 \ 0010), (000 \ 10 \ 010 \ 011 \ 11) \> \leq \< \lambda, \mu \>. \qedhere
\]
\end{proof}

\begin{proof}[Proof of Corollary~\ref{cor:torsion_v}]
Let $\lambda = (00 \ 01)$ and $\mu = (000 \ 0010 \ 0011 \ 10)(010 \ 011 \ 11)$, and write $\nu = \mu^{-1}\lambda$.
In terms of the elements in the statement, note that $\phi = \nu_{[0]}$ and $\psi = \mu_{[1]}$.

Let $\rho = \lambda_{[010]}$ and $\sigma = \mu_{[010]}$, so $\tau = \sigma^{-1}\rho = \nu_{[010]}$, and let $C = 010\C$.
We claim that $(\sigma, \tau) \in \S_{2, 112}(V_C)$, noting that $112 = r(2,3)$.
To prove this, first note that $\sigma\tau = \rho = \lambda_{[010]}$, which is an element of small support and order $2$.
Next note that 
\[
\< \sigma, \tau \> = \< \rho, \sigma \> = \< \lambda_{[010]}, \mu_{[010]} \>_{[010]} = V_C,
\]
where the final equality holds by Lemma~\ref{lem:torsion_v}.
It remains to prove that $[\sigma^{112}, \tau^{112}] \neq 1$.
Note that $112 \equiv 0 \pmod{4}$ and $112 \equiv 1 \pmod{3}$.
In particular, $\mu^{112} = (010 \ 011 \ 11)$.
Under the action of $\nu = \mu^{-1}\lambda$,
\[
10\bar{1} \mapsto 0\bar{1} \mapsto 000\bar{1} \mapsto 10\bar{1},
\]
so $\nu^{112}$ maps $10\bar{1}$ to $0\bar{1}$.
Therefore, $[\mu^{112},\nu^{112}]$ maps $\bar{1}$ to $0\bar{1}$, so $[\mu^{112},\nu^{112}]$ is nontrivial.
As a result, $[\sigma^{112}, \tau^{112}]$ is nontrivial, as required.
This proves the claim.

Therefore, it remains to prove that $(\alpha, \beta)$ in the statement is (a valid choice of) the pair $\Phi_{2, 3}(V, C, \sigma, \tau)$. 
To do this, it suffices to give choices for the elements and clopen sets featuring in Definition~\ref{def:technical} that line up with the description of the elements in the statement.

Define $E = 00\C$ and recall that $C = 010\C$. Define
\[
\begin{array}{lll}
\gamma_{00}  = 1               & \gamma_{01}  = (010 \ 100   )  & \gamma_{02}  = (010 \ 110   ) \\
\gamma_{10}  = (010 \ 0110  )  & \gamma_{11}  = (010 \ 1010  )  & \gamma_{12}  = (010 \ 1110  ) \\
\gamma_{20}  = (010 \ 0111  )  & \gamma_{21}  = (010 \ 1011  )  & \gamma_{22}  = (010 \ 1111  ) \\
\gamma_{001} = (00  \ 010100)  & \gamma_{010} = (00  \ 010110)  & \gamma_{011} = (00  \ 010111).
\end{array}
\]
Then, the element $\beta$ in Definition~\ref{def:technical} is
\begin{gather*}
(C\gamma_{00} \ C\gamma_{01} \ C\gamma_{02})(C\gamma_{10} \ C\gamma_{11} \ C\gamma_{12})(C\gamma_{20} \ C\gamma_{21} \ C\gamma_{22}) \\
= (010 \ 100 \ 110)(0110 \ 1010 \ 1110)(0111 \ 1011 \ 1111) 
= (01 \ 10 \ 11) = \beta.
\end{gather*}
In addition,
\begin{gather*}
\pi_1 = (C\gamma_{01} \ C\gamma_{10})(C\gamma_{02} \ C\gamma_{20}) = (100 \ 0110)(110 \ 0111) \\
\pi_2 = (C\gamma_{12} \ C\gamma_{11})(C\gamma_{22} \ C\gamma_{21}) = (1110 \ 1010)(1111 \ 1011) = (101 \ 111) \\
\pi_3 = (E \ E\gamma_{001})(E\gamma_{010} \ E\gamma_{011}) = (00 \ 010100)(010110 \ 010111) \\
\tau^{-\gamma_{11}} \tau^{\gamma_{12}} \sigma^{-\gamma_{21}} \sigma^{\gamma_{22}} 
= (\nu_{[1010]})^{-1} \nu_{[1110]} (\mu_{[1011]})^{-1} \mu_{[1111]} 
= (\phi\psi)^{-1}_{[101]}(\phi\psi)_{[111]}.
\end{gather*}
and bringing this together, the element $\alpha$ in Definition~\ref{def:technical} is
\begin{gather*}
(\phi\psi)^{-1}_{[101]}(\phi\psi)_{[111]} (100 \ 0110)(110 \ 0111)  (101 \ 111)  (00 \ 010100)(010110 \ 010111) \\ 
= (00 \ 010100)(010110 \ 010111)(100 \ 0110)(110 \ 0111) \cdot (\phi\psi)^{-1}_{[101]}(\phi\psi)_{[111]} (101 \ 111) = \alpha.
\end{gather*}
In this way, we have obtained the elements in the statement.
\end{proof}

\section{Uniform spread} 
\label{s:spread}

\subsection{Lower bounds}
\label{ss:spread_lower}

We now give our results on uniform spread, beginning with a strong version of Theorem~\ref{thm:spread}.

\begin{proposition}
\label{prop:spread}
Let $G \leq \Homeo(\C)$ be a finitely generated simple vigorous group. 
Then there exists an element $\sigma \in G$ of small support and order 30 such that for any nontrivial element $\alpha \in G$ there exists $\beta \in \sigma^G$ such that $\< \alpha, \beta \> = G$.
\end{proposition}

\begin{proof}
Fix a clopen set $\emptyset \subsetneq C \subsetneq \C$. 
Since $G$ is vigorous, we may fix $\gamma \in G$ such that $C \cap C\gamma = \emptyset$ and $C \cup C\gamma \subsetneq \C$. 
Write $\C = C \sqcup C\gamma \sqcup D$. 
By Lemma~\ref{lem:restriction}, $G_C \in \mathcal{K}_C^{\text{fg}}$ since $G \in \mathcal{K}^{\text{fg}}$. 

By Theorem~\ref{thm:torsion_involutions_strong}(i), there exist $\mu, \nu \in G_C$ of small support such that $|\mu| = 2$, $|\nu| = 5$ and $G_C = \< \mu, \nu \>$. 
Fix $\gamma_1, \gamma_2 \in G$ such that $D\gamma_1 \subsetneq \Fix_C(\mu)$ and $D\gamma_2 \subsetneq \Fix_{C\gamma}(\nu^\gamma)$. 
By Lemma~\ref{lem:existence}, we can fix $\lambda \in G$ that enacts the 3-cycle $(D\gamma_1 \ D\gamma_2 \ D)$. 
Define $\sigma = \lambda\mu\nu^{\gamma}$. 
Notice that $\lambda$, $\mu$ and $\nu^\gamma$ have coprime order and disjoint support, so each of these elements is a suitable power of $\sigma$. 
Observe also that $|\sigma| = |\lambda||\mu||\nu| = 30$. 
Moreover, $\sigma$ is an element of small support since it fixes $\Fix_C(\mu) \setminus D\gamma_1 \neq \emptyset$.

Let $\alpha \in G$ be an arbitrary nontrivial element. 
Fix a clopen set $\emptyset \subsetneq U_0 \subsetneq \C$ such that $U_0 \cap U_0\alpha = \emptyset$ and $U_0 \cup U_0\alpha \subsetneq \C$. 
Since $G$ is vigorous, there exists $\tau_1 \in G$ such that $C\tau_1 \subseteq U_0$. 
Let $U = C\tau_1$. 
Note that $U \cap U\alpha = \emptyset$ and $U \cup U\alpha \subsetneq \C$. 
Note also that $U\alpha = (C\gamma)\tau_2$ where $\tau_2 = \gamma^{-1}\tau_1\alpha$. 
Write $\C = U \sqcup U\alpha \sqcup V$. 
By Lemma~\ref{lem:partition}, 
\[
[C]_G + [C\gamma]_G + [D]_G = [\C]_G= [U]_G + [U\alpha]_G + [V]_G.
\] 
Now $[C]_G = [C\gamma]_G = [U]_G = [U\alpha]_G$, so $[D]_G = [V]_G$. 
In particular, there exists $\tau_3 \in G$ such that $V = D\tau_3$. 
Since $G$ is approximately full, there exists $\tau \in G$ such that $\tau|_C = \tau_1|_C$ and $\tau|_{C\gamma} = \tau_2|_{C\gamma}$; it follows that $D\tau = V$. 

Let $\beta = \sigma^\tau$. 
We claim that $\< \alpha, \beta \> = G$. 
First note that $\mu^\tau, \nu^{\gamma\tau} \in \< \beta \>$ and 
\[
\< \alpha, \beta \> \geq \< \mu^\tau, \nu^{\gamma\tau\alpha^{-1}} \> = \< \mu^\tau, \nu^\tau \> = G_C^\tau = G_U,
\] 
noting that $\nu^{\gamma\tau\alpha^{-1}} = \nu^{\gamma\tau_2\alpha^{-1}} = \nu^{\tau_1} = \nu^\tau$. 
Next note that $\lambda^\tau \in \< \beta \>$ enacts the permutation $(D\gamma_1\tau \ D\gamma_2\tau \ D\tau)$. 
Fix a clopen set $D\gamma_1\tau \subsetneq W \subsetneq \Fix_U(\mu^\tau)$; since $W \leq U$, we know that $G_W \leq G_U \leq \< \alpha, \beta \>$. 
Note that 
\[
\C = U \cup W \cup W(\lambda^\tau)^2 \cup W\lambda^\tau \cup U\alpha
\]
where adjacent terms in this union are not disjoint; in particular, note 
\[
W(\lambda^\tau)^2 \supseteq D\gamma_1\tau(\lambda^\tau)^2 = D\tau = V.
\] 
Therefore, repeated application of Lemma~\ref{lem:union}, implies that $G = \< G_U, G_W, \alpha, \lambda^\tau \> \leq \< \alpha, \beta \>$.
\end{proof}

Next we focus on Thompson's group $V$ and prove a strong version of Theorem~\ref{thm:spread_two_swaps}.

\begin{proposition} 
\label{prop:spread_two_swaps}
There exists an element $\sigma \in V$ such that $\< (00 \ 01), \sigma \> = \< (10 \ 11), \sigma \> = V$.
Moreover, $\sigma$ can be chosen to have order $7$.
\end{proposition}

To prove Proposition~\ref{prop:spread_two_swaps}, we first establish a slight strengthening of Theorem~\ref{thm:involutions} for $V$.

\begin{lemma} \label{lem:involutions_v_support}
There exist involutions $\alpha, \beta, \gamma \in V$ such that $\Supp(\alpha) \subsetneq \Supp(\beta) = \Supp(\gamma) = \C$ and $V = \< \alpha, \beta, \gamma \>$.
\end{lemma}

\begin{proof}
By Corollary~\ref{cor:torsion_v}, there exists an involution $\alpha \in V$ such that $\Supp(\alpha) \subsetneq \C$ and $\< \alpha, (01 \ 10 \ 11) \> = V$.
Therefore, it suffices to find involutions $\beta, \gamma \in V$ such that $\Supp(\beta) = \Supp(\gamma) = \C$ and $\beta\gamma = (01 \ 10 \ 11)$. 
One such choice is
\begin{gather*}
\beta  = (000 \ 001)(010 \ 111)(100 \ 101)(110 \ 011) \\
\gamma = (000 \ 001)(010 \ 011)(100 \ 111)(110 \ 101). \qedhere
\end{gather*}
\end{proof}

\begin{proof}[Proof of Proposition~\ref{prop:spread_two_swaps}]
To begin with, fix $\rho \in V$ such that $|\rho| = 7$ and $\Supp(\rho) = \C$.
The proof proceeds in two steps.
First, we will define involutions $\beta, \gamma \in V$ such that $\Supp(\beta) \sqcup \Supp(\beta^\gamma) = \C$ and $\rho^\gamma = \rho^{-1}$.
Second, we will prove that $\< \beta, \rho \> = V$. 
Since $\< \beta^\gamma, \rho \> = \< \beta^\gamma, \rho^{-1} \> = \< \beta^\gamma, \rho^\gamma \> = \< \beta, \rho \>^\gamma$, these steps also prove that $\< \beta^\gamma, \rho \> = V$.
Let us explain why these two steps prove the proposition.
By construction, $\beta$ and $\beta^\gamma$ are involutions such that $\Supp(\beta) \sqcup \Supp(\beta^\gamma) = \C$. 
Since $(00 \ 01)$ and $(10 \ 11)$ are also involutions such that $\Supp((00 \ 01)) \sqcup \Supp((10 \ 11)) =  \C$, there exists $\tau \in V$ such that $\beta^\tau = (00 \ 01)$ and $(\beta^\gamma)^\tau = (10 \ 11)$, and hence $\< (00 \ 01), \rho^\tau \> = \< (10 \ 11), \rho^\tau \> = V$, so taking $\sigma = \rho^\tau$ gives our desired result. 

As the first step, we define $\beta$ and $\gamma$.
To do this, we need to introduce a number of intermediate elements.
Fix a clopen subset $U \subseteq \C$ such that $\C = \bigsqcup_{i=0}^6 U\rho^i$. 
Since $V_U \cong V$, by Lemma~\ref{lem:involutions_v_support}, we may fix involutions $\alpha_0, \alpha_1, \alpha_2 \in V_U$ such that $\Supp(\alpha_0) \subsetneq \Supp(\alpha_1) = \Supp(\alpha_2) = U$ and $V_U = \< \alpha_0, \alpha_1, \alpha_2 \>$.
Write $U_0 = \Supp(\alpha_0)$ and $U_1 = U \setminus \Supp(\alpha_0)$. 
Let $\alpha_3 \in V$ be an involution such that $\Supp(\alpha_3) = U_0 \cup U_0\rho$ and $U_0\alpha_3 = U_0\rho$ (so, necessarily, $(U_0\rho)\alpha_3 = U_0$).
Let $\pi \in V$ be an involution such that $\Supp(\pi) = U$ and $U_0\pi = U_1$ (so, necessarily, $U_1\pi = U_0$).

Let $\gamma \in V$ be defined for all $0 \leq i \leq 6$ as 
\[
\gamma|_{U\rho^i} = (\pi\rho)^7\rho^{-2i}|_{U\rho^i},
\] 
noting that $\gamma$ is an involution satisfying $(U\rho^i)\gamma = U\rho^{-i}$ for all $i$.
In addition, for all $0 \leq i \leq 6$ and all $x \in U\rho^i$ we have
\[
x\rho^\gamma = x(\pi\rho)^7\rho^{-2i} \cdot \rho \cdot (\pi\rho)^7\rho^{-2(-i+1)} = x\rho^{-1},
\]
so $\rho^\gamma = \rho^{-1}$, as required.

Let $\beta$ be defined as the product
\[
\beta = \alpha_0\alpha_1^\rho\alpha_2^{\rho^2}\alpha_3^{\rho^3(\pi\rho)^7}.
\]
Observe that
\[
\Supp(\alpha_0) = U_0, \quad 
\Supp(\alpha_1^\rho) = U\rho, \quad 
\Supp(\alpha_2^{\rho^2}) = U\rho^2, \quad
\Supp(\alpha_3^{\rho^3(\pi\rho)^7}) = (U_1\rho^3 \cup U_1\rho^4).
\]
Since these four supports are disjoint, $\beta$ is a product of four commuting involutions and, hence, is an involution.
In addition,
\begin{align*}
\Supp(\beta) &= U_0 \cup U\rho^1 \cup U\rho^2 \cup U_1\rho^3 \cup U_1\rho^4 \\
\Supp(\beta^\gamma) &= U_1 \cup U\rho^6 \cup U\rho^5 \cup U_0\rho^4 \cup U_0\rho^3,
\end{align*}
so $\Supp(\beta) \sqcup \Supp(\beta^\gamma) = \C$, as required.

As the second step, we prove that $\< \beta, \rho \> = V$. 
Since $\C = \bigsqcup_{i=0}^6 U\rho^i$, we certainly have 
\[
\C = \bigcup_{i=0}^6 U\rho^i \cup \bigcup_{i=0}^6 U\beta^{\rho^{-3}}\rho^i.
\]
Since $U\beta^{\rho^{-3}} = U\alpha_3^{(\pi\rho)^7} = U_0 \cup U_1\rho$, adjacent sets in the following sequence
\[
U, \ U\beta^{\rho^{-3}}, \ U\rho, \ U\beta^{\rho^{-3}}\rho, \ \dots, \ U\rho^6, \ U\beta^{\rho^{-3}}\rho^6
\]
have nontrivial intersection.
Therefore, repeated application of Lemma~\ref{lem:union} implies that 
\[
\< V_U, V_U^{\beta^{\rho^{-3}}}, V_U^{\rho}, V_U^{\beta^{\rho^{-3}}\rho}, \dots, V_U^{\rho^6}, V_U^{\beta^{\rho^{-3}}\rho^6} \> = V.
\]
Therefore, it suffices to prove that $V_U \leq \< \beta, \rho \>$.

To do this, let us introduce a particular subgroup of $\< \beta, \gamma \>$, namely,
\[
G = \< \beta, \beta^{\rho^{-1}}, \beta^{\rho^{-2}} \>.
\]
It is easy to see that $G$ setwise stabilises both $U$ and $U\rho^{-1}$.
For brevity, let us write 
\[
W = U \cup U\rho^{-1}
\]
(as an aid to the reader, $W$ is a ``double $U$'').
We claim that $G''|_W = V_U$ (where $G''$ is the second derived subgroup of $G$).
To see this, first note that
\begin{align*}
G|_U &= \< \beta|_U, \beta^{\rho^{-1}}|_U, \beta^{\rho^{-2}}|_U \> = \< \alpha_0, \alpha_1, \alpha_2 \> = V_U \\
G|_{U\rho^{-1}} &= \< \beta|_{U\rho^{-1}}, \beta^{\rho^{-1}}|_{U\rho^{-1}}, \beta^{\rho^{-2}}|_{U\rho^{-1}} \> = \< \alpha_0^{\rho^{-1}}, \alpha_1^{\rho^{-1}} \>.
\end{align*}
In particular, $G|_U$ is simple, whereas $G|_{U^{\rho^{-1}}}$, being generated by two involutions, is a quotient of the infinite dihedral group and hence is metabelian. 
This means that
\[
G''|_U = V_U \quad\text{and}\quad G''|_{U\rho^{-1}} = 1,
\]
which proves the claim.

We next introduce a particular element of $\< \beta, \gamma \>$, namely, the commutator
\[
\delta = [\beta^{\rho^{-1}}, \beta^{\rho^3}].
\] 
We claim that $\Supp(\delta) \subseteq W$.
Since
\[
\beta^{\rho^{-1}} = \alpha_0^{\rho^{-1}}\alpha_1          \alpha_2^{\rho}  \alpha_3^{\rho^2(\pi\rho)^7}
\quad\text{and}\quad
\beta^{\rho^3}    = \alpha_0^{\rho^3}   \alpha_1^{\rho^4} \alpha_2^{\rho^5}\alpha_3^{\rho^{-1}(\pi\rho)^7}, 
\]
we calculate that
\[
\delta = [\alpha_1, \alpha_3^{\rho^{-1}(\pi\rho)^7}].
\]
In particular, $\Supp(\delta) \subseteq W$. 
We also claim that $\Supp(\delta) \cap U$ is nonempty.
To prove this, we will actually prove that $U_1\alpha_1^{-1} \subseteq \Supp(\delta)$.
To this end, let $x \in U_1\alpha_1^{-1}$. 
Then $x\alpha_1 \in U_1$, so $(x\alpha_1)\alpha_3^{\rho^{-1}(\pi\rho)^7} \in U_1\rho^{-1}$ and hence $(x\alpha_1)\alpha_3^{\rho^{-1}(\pi\rho)^7}$ is fixed by $\alpha_1$. 
This means that
\[
x\delta = (x\alpha_1\alpha_3^{\rho^{-1}(\pi\rho)^7})\alpha_1\alpha_3^{\rho^{-1}(\pi\rho)^7} = (x\alpha_1)\alpha_3^{\rho^{-1}(\pi\rho)^7}\alpha_3^{\rho^{-1}(\pi\rho)^7} = x\alpha_1.
\]
Since $\Supp(\alpha_1) = U$, we must have $x\alpha_1 \neq x$. 
Therefore, $x\delta \neq x$, so $x \in \Supp(\delta)$, as claimed.

Since $\Supp(\delta) \cap U \neq \emptyset$, we may fix $X \subseteq U$ such that $X \cap X\delta = \emptyset$.
Fix $\overline{\tau}_1, \overline{\tau}_2 \in V_U$ such that $\Supp(\overline{\tau}_1), \Supp(\overline{\tau}_2) \subseteq X$ and $[\overline{\tau}_1, \overline{\tau}_2] \neq 1$. 
Since $G''|_W = V_U$, we may fix $\tau_1, \tau_2 \in G''$ such that $\tau_1|_W = \overline{\tau}_1$ and $\tau_2|_W = \overline{\tau}_2$. 

We now introduce a final element of $\< \beta, \gamma \>$, namely,
\[
\epsilon = [[\delta, \tau_1], \tau_2].
\]
We claim that $\epsilon$ is nontrivial and $\Supp(\epsilon) \subseteq U$.

We begin by considering $[\delta, \tau_1]$. 
Note that $[\delta, \tau_1]$ acts trivially on $\C \setminus W$ since $\delta$ acts trivially on $\C \setminus W$ and $\tau_1$ (being an element of $G''$) setwise stabilises $\C \setminus W$.
Since $[\delta, \tau_1] =  (\tau_1^{-1})^\delta\tau_1$ and $\Supp([\delta, \tau_1]) \subseteq W$, we have $[\delta, \tau_1] = (\overline{\tau}_1^{-1})^\delta\overline{\tau}_1$. 

We now consider $[[\delta, \tau_1], \tau_2]$.
Note that $[[\delta, \tau_1], \tau_2]$ acts trivially on $\C \setminus W$ since $[\delta, \tau_1]$ acts trivially on $\C \setminus W$ (as noted above) and $\tau_2$ (being an element of $G''$) setwise stabilises $\C \setminus W$.
Similarly, $[[\delta, \tau_1], \tau_2]$ acts trivially on $W \setminus X$ since $[\delta, \tau_1] = (\overline{\tau}_1^{-1})^\delta\overline{\tau}_1$ setwise stabilises $W \setminus X$ (as $\Supp(\overline{\tau}_1) \subseteq X$ and $\Supp((\overline{\tau}_1^{-1})^\delta) \subseteq W \setminus X$) and $\tau_2$ acts trivially on $W \setminus X$. 
Therefore,
\[
\Supp(\epsilon) = \Supp([[\delta, \tau_1], \tau_2]) \subseteq X \subseteq U,
\] 
as claimed.
Moreover, 
\[
\epsilon = [[\delta, \tau_1], \tau_2] = [[\delta, \tau_1], \tau_2]|_X = [[\delta, \tau_1]|_X, \tau_2|_X] = [\overline{\tau}_1, \overline{\tau}_2],
\] 
which is nontrivial, by construction. 
This completes the proof of the claim.

As $\epsilon$ is a nontrivial element of $V_U$ and $G''|_U = V_U$, we see that 
\[
\< \epsilon^{G''} \> = \< \epsilon^{V_U} \> = V_U,
\]
since $V_U$ is simple.
Because $\epsilon \in \< \beta, \rho \>$ and $G'' \leq \< \beta, \rho \>$, we conclude that $V_U \leq \< \beta, \rho \>$, as required.
This completes the proof.
\end{proof}

\subsection{Upper bounds}
\label{ss:spread_upper}

We conclude by making some general observations about the uniform spread of infinite groups, with a particular focus on simple vigorous groups. 
The definition of uniform spread originated in the study of finite groups, and for infinite groups, it is useful to make a distinction between two different phenomena which both correspond to $u(G) = \infty$.

\begin{definition}
\label{def:infinite_uniform_spread}
Let $G$ be a group.
\begin{enumerate}
\item The group $G$ has \emph{weak infinite uniform spread} if for every nonnegative integer $k$ there exists $\sigma \in G$ such that for any nontrivial $\alpha_1, \dots, \alpha_k \in G$, there exists $\beta \in \sigma^G$ such that $\< \alpha_1, \beta \> = \dots = \< \alpha_k, \beta \>$
\item The group $G$ has \emph{strong infinite uniform spread} if there exists $\sigma \in G$ such that for any finite collection of nontrivial elements $\alpha_1, \dots, \alpha_k \in G$, there exists $\beta \in \sigma^G$ such that $\< \alpha_1, \beta \> = \dots = \< \alpha_k, \beta \>$.
\end{enumerate}
\end{definition}

Our key observation is the following.

\begin{proposition} 
\label{prop:strong_infinite}
Let $\Omega$ be a set and let $G \leq \Sym(\Omega)$. 
Assume that all of the following hold
\begin{enumerate}
\item every element of $G$ setwise stabilises a nonempty finite subset of $\Omega$
\item $G$ does not setwise stabilise any nonempty finite subset of $\Omega$
\item for every positive integer $k$ there exist nontrivial $x_1, \dots, x_k \in G$ with pairwise disjoint supports.
\end{enumerate}
Then $G$ does not have strong infinite uniform spread.
\end{proposition} 

\begin{proof}
Let $s \in G$. 
By (i), $s$ stabilises a finite subset $\Delta$ of $\Omega$, of size $m$, say. 
By (iii), we may fix nontrivial $x_1, \dots, x_{m+1} \in G$ with pairwise disjoint supports. 
Let $g \in G$. 
Since $|\Delta| \leq m$, we may fix $1 \leq i \leq m+1$ such that the support of $x_i$ is disjoint from $\Delta^g$. 
Then $\<x_i, s^g\>$ setwise stabilises $\Delta^g$. 
By (ii), $G$ does not setwise stabilise any finite set, so $\<x_i,s^g\> \neq G$. 
This proves that there does not exist $g \in G$ such that for all $1 \leq i \leq m+1$ we have $\<x_i,s^g\> = G$. 
Since $s$ was arbitrary, we have shown that $G$ does not have strong infinite uniform spread.
\end{proof}

\begin{corollary} 
\label{cor:strong_infinite}
Let $G \leq \Homeo(\C)$ be a finitely generated simple vigorous group. 
Assume that every element of $G$ has a finite orbit.
Then $G$ does not have strong infinite uniform spread.
\end{corollary}

\begin{proof}
By Proposition~\ref{prop:strong_infinite}, it suffices to verify that $G$ satisfies conditions (ii) and (iii) of Proposition~\ref{prop:strong_infinite} where $\Omega = \C$. 

For (ii), let $\emptyset \subsetneq \Delta \subseteq \C$ be finite. 
Let $C$ be a proper clopen subset of $\C$ containing $\Delta$. 
Since $G$ is vigorous, there exists $\gamma \in G$ such that $C\gamma \subseteq \C \setminus C$. 
Then $\Delta \gamma$ is disjoint from $\Delta$, so $G$ does not setwise stabilise $\Delta$. 

For (iii), let $k$ be a positive integer.
As usual, we may fix clopen sets $\emptyset \subsetneq U, V \subsetneq \C$ and elements $\gamma_{11}, \gamma_{12}, \gamma_{13}, \dots, \gamma_{k1}, \gamma_{k2}, \gamma_{k3}$ such that $\C = \bigsqcup_{i=1}^k (U\gamma_{i1} \sqcup U\gamma_{i2} \sqcup U\gamma_{i3}) \sqcup V$. 
Since $G$ is a simple vigorous group, it is vigorous and approximately full, so by Lemma~\ref{lem:existence}, for each $1 \leq i \leq k$ there exists $\delta_i \in G$ that enacts the permutation $(U\gamma_{i1} \ U\gamma_{i2} \ U\gamma_{i3})$. 
The elements $\delta_1, \dots, \delta_k$ have pairwise disjoint support.
\end{proof}

\begin{corollary} 
\label{cor:strong_infinite_v}
Thompson's group $V$ does not have strong infinite uniform spread.
\end{corollary}

\begin{proof}
By Corollary~\ref{cor:strong_infinite}, it suffices to verify that every element of $V$ has a finite orbit on $\C$, and this follows immediately from the existence of revealing pairs, as introduced by Brin in \cite[Section~10]{Brin04} (see \cite[Section~4]{BBLGHMS13} for a detailed explanation). 
The main idea is as follows. 
If $g$ has finite order, then every orbit of $g$ is finite.
Now assume that $g$ has infinite order.
Since $g$ has a revealing pair, in particular, $g$ admits a tree pair with an attractor and a domain of attraction. 
This ensures that there exists nonempty words $u,v \in \{0,1\}^*$ and a positive integer $m$ such that $(u\C)g^m= uv\C$, from which it follows that $u\overline{v}g^m = u\overline{v}$, so $\{ u\overline{v}, u\overline{v}g, \dots, u\overline{v}g^{m-1} \}$ is stabilised by $g$ (here $\overline{v}$ denotes the infinite string obtained by repeating $v$).
\end{proof}

The following example demonstrates that there do exist finitely generated simple vigorous groups that contain an element with no finite orbits.

\begin{example}
\label{ex:strong_infinite}
Let $\alpha \in \Homeo(\C)$ be the \emph{odometer}, which is the self-similar transformation defined for all $x \in \C$ as 
\[
(0x)\alpha= 1x \quad \text{and} \quad (1x)\alpha = 0(x\alpha),
\]
and let $A$ be the infinite cyclic group generated by $\alpha$ (see \cite[Example~2.3]{Nekrashevych04} where $\alpha$ is referred to as the \emph{adding machine}).
Consider the R\"over--Nekrashevych group $V_2(A)$ associated to $A$ (see \cite[Definition~9.7]{Nekrashevych04}).

Since $V$ acts transitively on the set $\{ u\C \mid u \in \{0,1\}^* \}$, as a subgroup of $\Homeo(\C)$, the group $V_2(A)$ is just $\< V, \alpha_{[0]} \>$ (using the notation introduced in \eqref{eq:restriction}).
However, $\alpha_{[0]} = (0 \ 1)\alpha$ since for all $x \in \C$ we have
\[
(0x)((0 \ 1)\alpha) = (1x) \alpha = 0(x\alpha) = (0x)\alpha_{[0]}  
\quad \text{and} \quad 
(1x)((0 \ 1)\alpha) = (0x) \alpha = 1x = (1x)\alpha_{[0]}.
\]
Therefore, $V_2(A) = \< V, \alpha_{[0]} \> = \< V, \alpha \>$.

Let $G$ be the derived subgroup of $V_2(A)$, and let $\gamma$ be the commutator $[(0 \ 1), \alpha_{[0]}] \in G$.
We claim that $G$ is a finitely generated simple vigorous group, and $\gamma$ has no finite orbits.

We begin by explaining why $\gamma$ has no finite orbits. 
First observe that $\overline{1}\alpha = \overline{0}$ while $\alpha$ restricts to a map $\C \setminus \{ \overline{1} \} \to \C \setminus \{ \overline{0} \}$ defined by the infinite set of prefix substitutions
\[
0x \mapsto 1x, \quad 10x \mapsto 01x, \quad 110x \mapsto 001x, \quad \dots \quad 1^n0x \mapsto 0^n1x, \quad \dots
\]
In particular, $\alpha$ has no finite orbits.
Now observe that
\[
\gamma = [(0 \ 1), \alpha_{[0]}] = (0 \ 1)^{-1} \alpha_{[0]}^{-1} (0 \ 1) \alpha_{[0]} = \alpha_{[1]}^{-1} \alpha_{[0]} = \alpha_{[0]} \alpha_{[1]}^{-1}.
\]
Put informally, $\gamma$ is the odometer on $0\C$ and the inverse-odometer on $1\C$.
In particular, $\gamma$ has no finite orbits.

We now turn to justifying our claims about $G$.
As $V$ is perfect, $V \leq G$, so $G$ is vigorous since $V$ is.
Moreover, by  \cite[Theorem~9.11]{Nekrashevych04}, all proper quotients of $V_2(A)$ are abelian, so $G$, being the derived subgroup of $V_2(A)$, is simple.
Finally, we claim that $G = \< V, \gamma \>$, which establishes that $G$ is finitely generated since $V$ is.

To prove this final claim, we need some groundwork.
Since $V$ acts by prefix substitutions on all points of $\C$, and $\alpha$ and $\alpha^{-1}$ each act by prefix substitutions on all but one point of $\C$, any element $\sigma \in V_2(A)$ acts by prefix substitutions on all but finitely many points on $\C$. 
Moreover, by inspecting the action of $\alpha$ and $\alpha^{-1}$, we see that if $\sigma \in V_2(A)$ does not act as a prefix substitution at $x \in \C$, then $x$ has an infinite tail that is $\overline{0}$ or $\overline{1}$.
For $\sigma \in V_2(A)$, let $\Phi_0(\sigma)$ (respectively, $\Phi_1(\sigma)$) be the set of points of $\C$ with an infinite tail $\overline{0}$ (respectively, $\overline{1}$) on which $\sigma$ does not act as a prefix substitution.
Let $\phi \: V_2(A) \to \Z$ be defined, for all $\sigma \in V_2(A)$, as 
\[
\sigma\phi = |\Phi_1(\sigma)| - |\Phi_0(\sigma)|.
\]
It is easy to check that $\phi$ is a homomorphism by considering the action $V_2(A)$ on $\C$.
Since $G$ is nonabelian simple, we must have $G \leq \ker\phi$, so $|\Phi_0(\sigma)| = |\Phi_1(\sigma)|$ for all $\sigma \in G$.

We are now in a position to prove that $G = \< V, \gamma \>$.
Let $\sigma \in G$. 
Since $|\Phi_0(\sigma)| = |\Phi_1(\sigma)|$, we can choose $x_1, \dots, x_k, y_1, \dots, y_k \in \{ 0,1 \}^*$ such that
\[
\Phi_0(\sigma) = \{ x_1 \overline{0}, \dots, x_k \overline{0} \} 
\quad \text{and} \quad 
\Phi_1(\sigma) = \{ y_1 \overline{1}, \dots, y_k \overline{1} \},
\]
For each nonnegative integer $i$, let $u_i$ be the binary expression for $i$ written in reverse and let $v_i$ be the complement of $u_i$, so 
\[
(u_i)_{i \geq 0} = (0, 1, 01, 11, 001, 101, \dots) 
\quad \text{and} \quad 
(v_i)_{i \geq 0} = (1, 0, 10, 00, 110, 010, \dots).
\]
Then
\[
\Phi_0(\gamma^k) = \{ 1u_0\overline{0}, 1u_1\overline{0}, \dots, 1u_{k-1}\overline{0} \}
\quad \text{and} \quad 
\Phi_1(\gamma^k) = \{ 0v_0\overline{1}, 0v_1\overline{1}, \dots, 0v_{k-1}\overline{1} \}.
\]
Fix $\delta \in V$ such that for all $1 \leq i \leq k$ we have
\[
(1 u_{i-1} \overline{0}) \delta = x_i \overline{0}
\quad \text{and} \quad 
(0 v_{i-1} \overline{1}) \delta =  y_i \overline{1}
\]
By construction, for all $x \in \C$, if $\sigma$ acts as a prefix substitution on $x$, then so does $(\gamma^k)^\delta$, and otherwise $\sigma((\gamma^k)^\delta)^{-1}$ acts as a prefix substitution on $x$.
This proves that $\sigma((\gamma^k)^\delta)^{-1} \in V$.
Therefore, $\sigma \in \< V, \gamma \>$, as required.

Alternatively, we could have defined $G$ as $\< V, \gamma \>$, noted that $G = \< V, [(0 \ 1), \alpha_{[00]}] \>$ and applied \cite[Proposition 5.5]{BleakElliottHyde24} to deduce that $G$ is a finitely generated simple vigorous group since $V$ is.
\end{example}

In light of Proposition~\ref{prop:strong_infinite}, we conclude with the following provocative question that refines Question~\ref{que:spread} from the introduction.

\begin{question} 
\label{que:spread_refined}
Let $G \leq \Homeo(\C)$ be a finitely generated simple vigorous group. 
\begin{enumerate}
\item Does $G$ have weak infinite uniform spread if every element has a finite orbit?
\item Does $G$ have strong infinite uniform spread if there is an element with no finite orbit?
\end{enumerate}
\end{question}


\begin{acknowledgements*}
The authors thank Jim Belk for helpful conversations and Colva Roney-Dougal for asking a question that initially prompted Theorem~\ref{thm:products}.
The third author held an EPSRC Postdoctoral Fellowship (EP/X011879/2). 
The work in this paper emerged from the Focussed Research Workshop \emph{Generating Thompson Groups} sponsored by the Heilbronn Institute for Mathematical Research. 
An initial draft of this paper was given to Claude (Opus~4.8): it identified a few typographical errors and some minor mathematical corrections, and we have implemented these in this version.
In order to meet institutional and research funder open access requirements, any accepted manuscript arising shall be open access under a Creative Commons Attribution (CC BY) reuse licence with zero embargo.
No new data were created.
\end{acknowledgements*}

\vspace{2pt}

\begin{multicols}{2}
\noindent Collin Bleak \newline
School of Mathematics and Statistics \newline
University of St Andrews \newline
St Andrews, KY16 9SS, UK \newline
\texttt{collin.bleak@st-andrews.ac.uk}
\vspace{5pt}

\noindent Casey Donoven \newline
College of Arts, Sciences \& Education \newline
Montana State University-Northern \newline
Havre, Montana 59501, USA \newline
\texttt{casey.donoven@msun.edu}

\noindent Scott Harper \newline
School of Mathematics \newline
University of Birmingham \newline
Birmingham, B15 2TT, UK \newline
\texttt{s.harper.3@bham.ac.uk}
\vspace{5pt}

\noindent James Hyde \newline
Department of Mathematics and Statistics \newline
Binghamton University \newline
Binghamton, New York 13902, USA \newline
\texttt{jhyde1@binghamton.edu}
\end{multicols}

\begin{thebibliography}{99}

\bibitem{AschbacherGuralnick84}
M. Aschbacher and R. Guralnick, \emph{Some applications of the first cohomology group}, J. Algebra \textbf{90} (1984), 446--460.

\bibitem{BelkZaremsky22}
J. Belk and M. C. B. Zaremsky, \emph{Twisted Brin--Thompson groups}, Geom. Topol. \textbf{26} (2022), 1189--1223.

\bibitem{BBLGHMS13}
C. Bleak, H. Bowman, A. G. Lynch, G. Graham, J. Hughes, F. Matucci and E. Sapir, \emph{Centralizers in the R. Thompson group $V_n$}, Groups Geom. Dyn. \textbf{7} (2013), 821--865.

\bibitem{BleakElliottHyde24} 
C. Bleak, L. Elliott and J. Hyde, \emph{Sufficient conditions for a group of homeomorphisms of the Cantor set to be two-generated}, J. Inst. Math. Jussieu \textbf{23} (2024), 2825--2858.

\bibitem{BleakQuick17}
C. Bleak and M. Quick, \emph{The infinite simple group $V$ of Richard J. Thompson: presentations by permutations}, Groups Geom. Dyn. \textbf{11} (2017), 1401--1436.

\bibitem{BooneHigman74}
W. W. Boone and G. Higman, \emph{An algebraic characterization of groups with soluble word problem}, J. Austral. Math. Soc. \textbf{18} (1974), 41--53.

\bibitem{Bridson98}
M. R. Bridson, \emph{Controlled embeddings into groups that have no non-trivial finite quotients}, in \emph{The Epstein birthday schrift}, Geom. Topol. Mongr., vol.~1, Geometry \& Topology Publications, 1998, 99--116.

\bibitem{Brin04}
M. G. Brin, \emph{Higher dimensional Thompson groups}, Geom. Dedicata \textbf{108} (2004), 163--192.

\bibitem{BrownGeoghegan85}
K. S. Brown and R. Geoghegan, \emph{Cohomology with free coefficients of the fundamental group of a graph of groups}, Comment. Math. Helv. \textbf{60} (1985), 31--45.

\bibitem{Burness19}
T. C. Burness, \emph{Simple groups, generation and probabilistic methods}, in \emph{Groups St Andrews 2017 in Birmingham}, London Math. Soc. Lecture Note Ser., vol.~455, Cambridge Univ. Press, 2019, 200--229.

\bibitem{BurnessGuralnickHarper21}
T. C. Burness, R. M. Guralnick and S. Harper, \emph{The spread of a finite group}, Ann. of Math. \textbf{193} (2021), 619--687.

\bibitem{DonovenHarper20}
C. Donoven and S. Harper, \emph{Infinite $\frac{3}{2}$-generated groups}, Bull. Lond. Math. Soc. \textbf{52} (2020), 657--673.

\bibitem{GhysSergiescu87}
{\'E}. Ghys and V. Sergiescu, \emph{Sur un groupe remarquable de diff\'eomorphismes du cercle}, Comment. Math. Helv. \textbf{62} (1987), 185--239.

\bibitem{Goryushkin74}
A. P. Goryushkin, \emph{Imbedding of countable groups in $2$-generated simple groups}, Mathematical Notes of the Academy of Sciences of the USSR \textbf{16} (1974), 725--727.

\bibitem{GuralnickKantor00} 
R. M. Guralnick and W. M. Kantor, \emph{Probabilistic generation of finite simple groups}, J. Algebra \textbf{234} (2000), 743--792.

\bibitem{Hall74}
P. Hall, \emph{On the embedding of a group in a join of given groups}, J. Austral. Math. Soc. \textbf{17} (1974), 434--495.

\bibitem{Harper23}
S. Harper, \emph{The maximal size of a minimal generating set}, Forum Math. Sigma \textbf{11} (2023), e70 1--10.

\bibitem{Harper24Survey}
S. Harper, \emph{The spread of finite and infinite groups}, in \emph{Groups St Andrews 2022 in Newcastle}, London Math. Soc. Lecture Note Series, vol.~496, Cambridge Univ. Press, 2024, 74--117.

\bibitem{LiebeckShalev96}
M. W. Liebeck and A. Shalev, \emph{Classical groups, probabilistic methods, and the $(2,3)$-generation problem}, Ann. of Math. \text{144} (1996), 77--125.

\bibitem{LubeckMalle99}
F. L\"ubeck and G. Malle, \emph{$(2,3)$-generation of exceptional groups}, J. Lond. Math. Soc. \textbf{59} (1999), 109--122.

\bibitem{MalleSaxlWeigel94}
G. Malle, J. Saxl and T. Weigel, \emph{Generation of classical groups}, Geom. Dedicata \textbf{49} (1994), 85--116.

\bibitem{Matui12}
H. Matui, \emph{Homology and topological full groups of \'etale groupoids on totally disconnected spaces}, Proc. London Math. Soc. \textbf{104} (2012), 27--56.

\bibitem{McKenzieThompson73}
R. McKenzie and R. J. Thompson, \emph{An elementary construction of unsolvable word problems in group theory}, in \emph{Word problems: decision problems and the Burnside problem in group theory (Conf., Univ. California, Irvine, Calif. 1969)}, Studies in Logic and the Foundations of Math., vol.~71, North-Holland, 1973, 457--478.

\bibitem{Miller01}
G. A. Miller, \emph{On the groups generated by two operators}, Bull. Amer. Math. Soc. \textbf{7} (1901), 424--426.

\bibitem{Nekrashevych04}
V. V. Nekrashevych, \emph{Cuntz-Pimsner algebras of group actions}, J. Operator Theory \textbf{52} (2004), 223--249.

\bibitem{Sapir}
M. Sapir, personal written communication with Bleak (2017) upon reviewing \cite{BleakQuick17}.

\bibitem{ScheslerSkipperWu}
E. Schesler, R. Skipper and X. Wu, \emph{The Higman--Thompson groups $V_n$ are $(2,2,2)$-generated}, preprint, \url{arxiv:2411.09069}, 2024.

\bibitem{Schupp76}
P. E. Schupp, \emph{Embeddings into simple groups}, J. Lond. Math. Soc. \textbf{13} (1976), 90--94.

\bibitem{Scott92}
E. A. Scott, \emph{A tour around finitely presented infinite simple groups}, in \emph{Algorithms and Classification in Combinatorial Group Theory}, Mathematical Sciences Research Institute Publications, vol.~23, Springer--Verlag, 1992, 83--119.

\bibitem{Steinberg62}
R. Steinberg, \emph{Generators for simple groups}, Canadian J. Math. \textbf{14} (1962), 277--283.

\bibitem{Tarski75}
A. Tarski, \emph{An interpolation theorem for irredundant bases of closure structures}, Discrete Math. \textbf{12} (1975), 185--192.

\bibitem{ThompsonHandwritten}
R. J. Thompson, widely circulated handwritten notes  (1965), 1--11.

\bibitem{Thompson76}
R. J. Thompson, \emph{Embeddings into finitely generated simple groups which preserve the word problem}, in \emph{Word problems, {II} (Conf. on Decision Problems in Algebra, Oxford, 1976)}, Studies in Logic and the Foundations of Math., vol.~95, North-Holland, 1980, 401--441.

\bibitem{Whiston00}
J. Whiston, \emph{Maximal independent generating sets of the symmetric group}, J. Algebra \textbf{232} (2000), 255--268.

\bibitem{WiegoldWilson78}
J. Wiegold and J. S. Wilson, \emph{Growth sequences of finitely generated groups}, Arch. Math. \textbf{30} (1978), 337--343.

\bibitem{Zaremsky24} 
M. C. B. Zaremsky, \emph{Finite presentability of twisted Brin--Thompson groups}, Proc. Roy. Soc. Edinburgh Sect. A, appears online (2024), 1--12.

\end{thebibliography}
\end{document}